\documentclass{birkjour}
\usepackage[dvipsnames]{xcolor}
\usepackage{amsfonts,amssymb,amsmath,mathtools,hyperref}
\usepackage[english]{babel}
\usepackage{bigints}
\usepackage{amssymb}
\usepackage{pdfsync}
\usepackage[normalem]{ulem}
\usepackage{cancel}

\usepackage{fixltx2e}
\usepackage[all,error]{onlyamsmath}

\usepackage{amssymb,amsfonts,amsmath,
amsrefs,enumerate,url,hyperref,
mathbbol,dsfont}
\usepackage{mathrsfs}
\usepackage{tikz}
\usetikzlibrary{quotes,angles}
\usetikzlibrary{arrows.meta}

\tikzset{%
  >=latex, % option for nice arrows
  inner sep=0pt,%
  outer sep=2pt,%
  mark coordinate/.style={inner sep=0pt,outer sep=0pt,minimum size=3pt,
    fill=black,circle}%
}
\usepackage[T1]{fontenc}
\usepackage[cp1250]{inputenc}

\usepackage[strict=true]{csquotes}

\numberwithin{equation}{section}

\usepackage{relsize,bigints}

\usepackage{graphicx,tikz}

\usepackage{color}
\newtheorem{thm}{Theorem}[section]
\newtheorem{lemma}[thm]{Lemma}
\newtheorem{proposition}[thm]{Proposition}

\newtheorem{remark}[thm]{Remark}

\newcommand{\eps}{\varepsilon}

\DeclareMathAlphabet{\mathpzc}{OT1}{pzc}{m}{it}

\newcommand{\ka}{{\mathpzc k}}
\newcommand{\kad}{{\mathpzc K}}

\newcommand{\pol}{\tfrac 12}

\newcommand{\mquad}[1]{\qquad \text{#1} \qquad}
\newcommand{\mqquad}[1]{\quad \text{#1} \quad}

\newcommand{\fir}{ f_i^{\mathpzc r}}

\newcommand{\e}{\mathrm e}
\newcommand{\ud }{\, \mathrm d}

\newcommand{\sem}[1]{\left ( \e^{t#1} \right)_{t \ge 0}}
\newcommand{\dom}[1]{\mathcal D(#1)}
\newcommand{\grae}{\lim_{\eps \to 0}}
\newcommand{\lam}{\lambda}
\newcommand{\rla}{R_\lam} 
\newcommand{\rlae}{R_{\lam,\eps}}

\newcommand{\jcg}[1]{\left (#1 \right )_{n\ge 1}} 
\newcommand{\gra}{\lim_{n\to \infty}}

\newcommand{\rez}[1]{\left (\lam - #1\right )^{-1}}

\newcommand{\coss}{\mbox{$C(S)$}}
\newcommand{\ie}{i.e., }
\newcommand{\czn}{\int_0^\infty}
\newcommand{\sui}{\sum_{i\in \kad}}
\newcommand{\cadlag}{c\`adl\`ag}
\newcommand{\seq}[1]{\left ( #1\right )_{i\in \kad}} 
\newcommand{\limx}{\lim_{x\to \infty}}

\newcommand{\rod}[1]{\mbox{$\left (#1 \right )_{t\ge 0}$}} 
\newcommand{\slam}{\sqrt{2\lam}}
\newcommand{\slem}{\sqrt{\tfrac 2\lam}}

\usepackage{graphicx,tikz}
\usetikzlibrary{arrows}
\definecolor{ngreen}{HTML}{006400}

\DeclareMathAlphabet{\mathpzc}{OT1}{pzc}{m}{it}

\newcommand{\jeden}{\mathbf 1}

\newcommand{\elam}{\ell_\lam} 
\newcommand{\elame}{\ell_{\lam,\eps}} 
\newcommand{\zero}{\mathpzc {o}}
\title[]{Perturbing Walsh's spider process by \\ infinitesimal drifts}

\author[A. Bobrowski]{Adam Bobrowski}
\address{Lublin University of Technology\\Department of Mathematics\\Nadbystrzycka 38A\\20-618 Lublin, Poland}
\email{a.bobrowski@pollub.pl}
\author[A. Pilipenko]{Andrey Pilipenko}
\address{University of Geneva, 
Section de math\'ematiques\\
Case postale 64\\
1211 Geneva 4, Switzerland\\ \\
Institute of Mathematics of the Ukrainian National Academy of Sciences\\
 Tereschenkivska st. 3, Kiev-4, 01601, Ukraine}
\email{Andrey.Pilipenko@unige.ch} 
\author[E. Ratajczyk]{El\.zbieta Ratajczyk}
\address{
Lublin University of Technology\\
Nadbystrzycka 38A\\
20-618 Lublin, Poland}

\email{e.ratajczyk@pollub.pl}

\begin{document}

\begin{abstract}We consider Walsh's process on a star graph perturbed by the fact that its paths need to additionally pass through semi-permeable soft membranes, each located on a different edge at a small distance from the center, and each having its own permeability coefficient. We describe  the limit processes as the distances converge to $0$, in terms of  the characteristics of the approximating ones. The main finding is that the limit process is still of Walsh type, but its parameters need to be appropriately modified. The paper is a close companion of \cites{onportenko,withvillmart}. \end{abstract}

\newcommand{\cer}{C[-\infty, \infty]}
\newcommand{\cerp}{C[0, \infty]}

\newcommand{\R}{\mathbb R} 
\newcommand{\gep}{\mathfrak G_\eps}
\newcommand{\gepo}{\mathfrak G_\eps^0}
\newcommand{\gen}{\mathfrak H_{\alpha,\beta}}
\newcommand{\genm}{\mathfrak H_{\wt \alpha,\wt \beta}}
\newcommand{\qep}{q_\eps}
\newcommand{\pep}{p_\eps}
\newcommand{\aep}{A_\eps}
\newcommand{\bep}{B_\eps}
\newcommand{\cem}{C_\eps^-}
\newcommand{\cep}{C_\eps^+}
\newcommand{\pepi}{p_{\eps,i}}
\newcommand{\abs}{\mathfrak H_{0,\infty}}

\maketitle

\newcommand{\wt}{\widetilde}
\section{Introduction} 
Walsh's spider process on a star graph with $\ka$ edges is a natural generalization of the skew Brownian motion. The latter has been introduced in the It\^o and McKean's book \cite{ito} as a modification of the standard one-dimensional Wiener process in which the sign of each excursion from $0$ is chosen to be positive with probability, say, $\alpha \in [0,1]$, and negative with probability $1-\alpha$. The process of Walsh can  be constructed similarly, and is characterized by $\ka$ non-negative parameters $\alpha_i, i = 1,\dots, \ka $ such that  $\sum_{i=1}^\ka  \alpha_i=1$; each excursion of the Brownian motion from $0$ is chosen to lie on the $i$th edge with probability $\alpha_i$. 

Interestingly, as shown by Portenko \cite{portenko,portenkokniga}, the skew Brownian motion can be seen as a process with singular drift at $0$. To wit, Portenko has proved that the skew Brownian motion can be obtained in the limit, as $\eps \to 0$,  of diffusions generated by the operators 
\[ \tfrac 12 \tfrac {\ud^2}{\ud x^2} + a_\eps \tfrac {\ud}{\ud x} \]
with drift coefficients concentrating around $0$, as, for instance,  in the example where $a_\eps (x) = \eps^{-1} a(\eps^{-1}x), x \in \R$ for a continuous integrable $a\colon \R \to \R$. Recently, in  \cite{onportenko}, Portenko's result has been extended to the case of Walsh's spider: we have shown that as a result of a singular perturbation of Walsh's process by drifts that concentrate more and more around the graph center as above, one obtains Walsh's process with transformed parameters. 

In this paper we take up the task of studying the related scenario in which Walsh's process is singularly perturbed by \emph{infinitesimal} drifts. To elaborate on this succinct statement, we recall that the skew Brownian motion can alternatively be seen as a process passing through a soft membrane located at $0$ (see e.g. \cite{bookandrey}). Our approximation involves $\ka$ membranes of this type, each located on a different edge of the star graph at a distance $c_i\eps>0 $ from its center, where $c_i>0, i=1,\ldots,k$ are given (see Figure \ref{rys1}). We study the limit as $\eps \to 0$, and describe the limit process in terms of the parameters characterizing the approximating membranes. We prove that such a~perturbation of the Walsh process is also a   Walsh process but with different parameters.

Before closing this introduction we note that our findings agree with those of the companion paper \cite{withvillmart}, where the case of $\ka =2$ was studied by other methods; see also \cite{makhnomem}.
Furthermore, just before sending our paper to the journal we discovered \cite{koralov2026}, where a criterion for convergence of diffusions on star graphs is discussed in terms of  scale functions and speed measures; by contrast, our approach is based on the classical theory of convergence of semigroups. 

\begin{figure}
  \begin{tikzpicture}[scale=1]
    \draw [thick] plot  coordinates {(0,0) (2,0)};
    \draw [thick] plot  coordinates {(0,0) (2*0.309,2*0.951)};
    \draw [thick] plot  coordinates {(0,0) (2*0.309,2*-0.951)};
    \draw [thick] plot  coordinates {(0,0) (-2*0.809,2*0.587)};
    \draw [thick] plot  coordinates {(0,0) (-2*0.809,-2*0.587)};
\draw[thick,densely dashdotted, RedOrange , opacity=0.8] (0,0) circle (0.1);
\draw[ densely dashdotted, RedOrange, opacity=0.6] (0,0) circle (0.15);
\node [below] at  (-0.1,-0.1) {\footnotesize{{$\zero$}}}; 
    \fill[black] (1,0) circle (1.7pt);
    \fill[NavyBlue] (1,0) circle (0.8pt);
\draw[thick,densely dashdotted, NavyBlue, opacity=0.8] (1,0) circle (0.1);
\draw[densely dashdotted, NavyBlue, opacity=0.4] (1,0) circle (0.15);
    \fill[black] (0.7*0.309,0.7*0.951) circle (1.7pt);
    \fill[NavyBlue] (0.7*0.309,0.7*0.951) circle (0.8pt);
\node [right] at  (1*0.309,0.7*0.951) {\small{{$c_2\varepsilon$}}}; 
\draw[thick,densely dashdotted, NavyBlue, opacity=0.8] (0.7*0.309,0.7*0.951) circle (0.1);
\draw[densely dashdotted, NavyBlue, opacity=0.4] (0.7*0.309,0.7*0.951) circle (0.15);
    \fill[black] (1.5*0.309,-1.5*0.951) circle (1.7pt);
    \fill[NavyBlue] (1.5*0.309,-1.5*0.951) circle (0.8pt);
\draw[thick,densely dashdotted, NavyBlue, opacity=0.8] (1.5*0.309,-1.5*0.951) circle (0.1);
\draw[densely dashdotted, NavyBlue, opacity=0.4] (1.5*0.309,-1.5*0.951) circle (0.15);
\node [left] at  (1.3*0.309,-1.55*0.951) {\small{{$c_5\varepsilon$}}}; 
    \fill[black] (-1.3*0.809,1.3*0.587) circle (1.7pt);
    \fill[NavyBlue] (-1.3*0.809,1.3*0.587) circle (0.8pt);
\draw[thick,densely dashdotted, NavyBlue, opacity=0.8] (-1.3*0.809,1.3*0.587) circle (0.1);
\draw[densely dashdotted, NavyBlue, opacity=0.4] (-1.3*0.809,1.3*0.587) circle (0.15);
\node [below] at  (-1.3*0.809,1.05*0.587) {\small{{$c_3\varepsilon$}}}; 
\node [below] at  (1,-0.15) {\small{{$c_1\varepsilon$}}}; 
    \fill[black] (-1*0.809,-1*0.587) circle (1.7pt);
    \fill[NavyBlue] (-1*0.809,-1*0.587) circle (0.8pt);
\draw[thick,  densely dashdotted, NavyBlue, opacity=0.8] (-1*0.809,-1*0.587) circle (0.1);
\draw[densely dashdotted, NavyBlue, opacity=0.4] (-1*0.809,-1*0.587) circle (0.15);
\node [below] at  (-1*0.809,-1.25*0.587) {\small{{$c_4\varepsilon$}}}; 
    \fill[black] (0,0) circle (1.7pt);
    \fill[RedOrange] (0,0) circle (0.8pt);
 \draw [thick] plot  coordinates {(4.2,0-1) (5.8,0-1)};
    \fill[black] (5,0-1) circle (1.7pt); 
    \fill[NavyBlue] (5,1.4) circle (0.8pt); 
\draw[thick,  densely dashdotted, NavyBlue, opacity=0.8] (5,0-1) circle (0.1);
\draw[densely dashdotted, NavyBlue, opacity=0.4] (5,0-1) circle (0.15);
\node [below] at  (5,0.45-1) {\small{{$c_i \varepsilon$}}}; 
\draw [NavyBlue,densely dashdotted,{Classical TikZ Rightarrow[scale=1]}- ,  thick] (4.5,-0.3-1) -- (5,-0.3-1);
\draw [NavyBlue,densely dashdotted,{Classical TikZ Rightarrow[scale=1]}- ,  thick] (5.5,-0.3-1) -- (5,-0.3-1);
    \fill[NavyBlue] (5,-0.3-1) circle (1.5pt);
\node [below] at  (5.6,-0.46-1) {\scriptsize{\textcolor{NavyBlue}{$p_{\varepsilon,i}$}}}; 
\node [below] at  (4.3,-0.42-1) {\scriptsize{\textcolor{NavyBlue}{$1-p_{\varepsilon,i}$}}}; 
   \draw [thick] plot  coordinates {(5,1.4) (5+0.4,1.4)};
    \draw [thick] plot  coordinates {(5,1.4) (5+0.4*0.309,1.4+0.4*0.951)};
    \draw [thick] plot  coordinates {(5,1.4) (5+0.4*0.309,1.4+0.4*-0.951)};
    \draw [thick] plot  coordinates {(5,1.4) (5-0.4*0.809,1.4+0.4*0.587)};
    \draw [thick] plot  coordinates {(5,1.4) (5-0.4*0.809,1.4-0.4*0.587)};
    \fill[black] (5,1.4) circle (1.7pt);
\draw[thick,densely dashdotted, RedOrange , opacity=0.8] (5,1.4) circle (0.1);
\draw[ densely dashdotted, RedOrange, opacity=0.6] (5,1.4) circle (0.15);
   \fill[RedOrange] (5,1.4) circle (0.8pt);
\node [below] at  (5-0.1,1.4-0.1) {\footnotesize{{$\zero$}}}; 
  \node[above] at (5,0.1) 
    {\scriptsize{$\displaystyle\sum_{i=1}^{\ka} \alpha_i f_i'(0)= \tfrac \beta 2 f''(\zero)$}};
%  \node[below] at (5,-2)   {\scriptsize{$ (1-\pepi) f'_i(\eps - )=\pepi f_i'(\eps+)$}};
  \end{tikzpicture}

\caption{\footnotesize{The star graph $S\coloneqq K_{1,\ka}$ with $\ka=5$ edges and soft semipermeable membranes located close to its center $\zero$. The behavior of the process at the graph center is described by the boundary condition in the upper right (see Section \ref{sec:tmr} for details). The  membrane at the $i$th edge is located at the $c_i \eps$ distance from $\zero$, and has its own permeability coefficient $\pepi$ which depends on $\eps>0$ and $i=1,\dots, \ka$; this coefficient describes the tendency of the membrane to reflect the process away from $\zero$. }}\label{rys1} 
\end{figure}
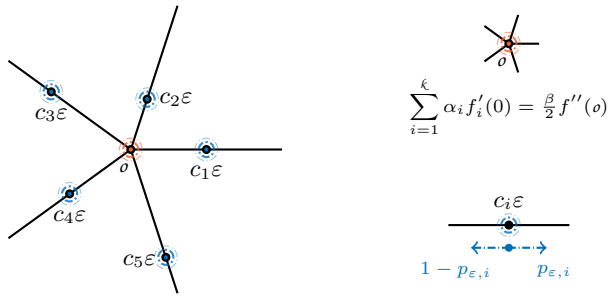

\section{Preliminaries} \label{sec:prelim}

\subsection{Feller semigroups and their convergence}

Let $S$ be a compact metric space, and $\coss$ be the space of continuous functions on $S$.  A Feller semigroup  (see e.g. \cites{bass,ethier,kallenbergnew}) 
is a strongly continuous family of positive contraction operators $\rod{T(t)}$  in $\coss$ such that $T(0)=I_{\coss}$ (the identity operator) and 
\[ T(t)T(s) = T(t+s) \qquad s,t \ge 0.\]
Because of the Riesz representation theorem, for each $x\in S $ and $t\ge 0$ there is a Borel measure $\textsf m_{x,t}$ on $S$ such that $\textsf m_{x,t} (S)\le 1$ and 
\begin{equation}T(t)f (x) = \int_S f \ud \textsf m_{x,t}, \qquad f \in \coss; \label{intro:1} \end{equation}
if $\textsf m_{x,t}(S) = 1$ for all $x$ and $t$, the process is said to be conservative or honest.

It is well-known that with each Feller semigroup one can associate a~Markov  process \rod{X(t)} with c\`adl\`ag paths \cite{ethier,kallenbergnew,rosen,rogers} (\ie paths that are right-continuous and possess left limits) in such a way that 
\begin{equation} \label{intro:2} T(t)f(x) = \mathsf E_x  f(X(t))\mathsf 1_{\{t< \tau\}} , \qquad  f\in \coss, t\ge 0, x\in S\end{equation}
where $\mathsf E_x $ denotes expectation conditional on $X(0)=x$ and $\tau$ is the lifetime of the process, that is, $\tau=\tau (\omega)$ is the random time up to which the path $X(t,\omega)$ is defined. Comparing Eqs. \eqref{intro:1} and \eqref{intro:2} we see that $\textsf m_{x,t}(\Gamma),$ {for a Borel set} $ \Gamma \subset S$ should be interpreted as the probability that  \rod{X(t)}  starting at $x$ at time $0$ will be in $\Gamma$ at a time $t\ge 0$. The measures $\textsf m_{x,t}$, together with the Markov property (which translates into Chapman--Kolmogorov equations), determine finite-dimensional distributions of \rod{X(t)}. If the distribution of $X(0)$ is specified, these finite-dimensional distributions, in turn, when combined with the requirement that the paths are \cadlag ,  
determine the distribution of \rod{X(t)} in the Skorokhod space $D([0,\infty), S)$.

Feller semigroups are conveniently described by their generators. The generator $\mathfrak A$ of a Feller semigroup is defined by 
\[\mathfrak A f = \lim_{t\to 0} t^{-1} (T(t)f - f) \]
on the domain $\dom{\mathfrak A }$ composed of $f$ such that the above strong limit exists; that is, the right-hand side converges to $\mathfrak A f$ uniformly on $S$. It is well-known that the generator $\mathfrak A$ characterizes the semigroup uniquely; in particular, different semigroups have different generators.  Moreover, an operator $\mathfrak A$ is the generator of a Feller semigroup (shortly: a Feller generator) in $\coss$ iff the following three conditions are met %the following three conditions are satisfied: 
\begin{itemize}
\item [1. ] $\mathfrak A$ is densely defined, 
\item [2. ] $\mathfrak A$ satisfies the positive-maximum principle,
\item [3. ] $\mathfrak A$ satisfies the range condition: for any $g \in \coss$ and $\lam > 0$ there is an $f \in \dom{\mathfrak A }$ such that $\lam f - \mathfrak A f = g.$ \end{itemize} 
The semigroup generated by $\mathfrak A$ will in what follows be denoted $\sem{\mathfrak A}.$ The solution $f$ to the resolvent equation is denoted $\rez{\mathfrak A} g$, and the family of operators $\rez{\mathfrak A}, \lam >0$ is termed the resolvent of the generator or the resolvent of the semigroup. 

In defining Feller semigroups it is customary to require also that the related process is honest, which comes down to the requirement that $\jeden_S \in \dom{\mathfrak A }$ and $\mathfrak A \mathsf 1_S= 0$ (where $\jeden_S \in \coss$ is the function equal to $1$ on the entire $S$).

We also recall the Trotter--Kato--Neveu--Macke\-vi\v cius theorem  (see e.g. \cite{kallenbergnew}*{p. 385} or  \cite{ethier}*{Chapter 4, Theorem 2.5}):  this classical result asserts that the  convergence of Feller semigroups is equivalent to the convergence of their Laplace transforms, that is, to the convergence of the resolvents of their generators. (See e.g. \cite{knigazcup} for a broader semigroup-theoretic perspective). Such convergence implies (and is implied by) weak convergence of distributions of the corresponding processes in the Skorokhod space $D([0,\infty), S)$, provided that the initial distributions of these processes converge weakly also. 
\subsection{Cosine families}  \label{sec:cf} 
A strongly continuous family $\{C(t), t \in \R\}$ of operators in a Banach space $\mathsf F$ is said to be a cosine family iff $C(0)$ is the identity operator and 
\[ 2 C(t) C(s) = C(s+t) + C(t-s), \qquad t,s\in \R.\]
The generator of such a family is defined by 
\[\mathfrak A f = \lim_{t\to 0} 2t^{-2}{(C(t)f- f)}\]
for all $f \in \mathsf F$ such that the limit on the right-hand side exists. 
For example, in $\cer$, the space of continuous functions $f\colon \R\to \R$ that have finite limits at $\pm \infty$,  there is the \emph{basic cosine family} given by  
\begin{equation}\label{intro:0} C(t) f(x) = \pol [ f(x+t) + f(x-t) ], \qquad x \in \R, t \in \R. \end{equation}
Its generator is the operator $f\mapsto \Delta f= f''$ with domain composed of twice continuously differentiable functions $f\in \cer$  such that $f'' \in \cer$. 

A cosine family is said to be composed of \emph{equibounded} operators if there is an $M>0$ such that 
\( \|C(t)\| \le M, t \in \R\); $M$ is called a bound. The basic cosine family falls into this category with bound $M=1$. As a result (cf.  \cite[Proposition 3.14.4]{abhn}), for all $f\in \cer$, the Laplace transform of the vector-valued function $t\mapsto C(t)f \in \cer$ is well-defined for $\lam >0$, $g\coloneqq \lam^{-1} \int_0^\infty \e^{-\lam t} C(t) f \ud t $ belongs to the domain of $\Delta$, and $\lam^2 g - \Delta g = f$. This can also be checked directly with the help of 
\[ g(x) = \tfrac 1{2\lam} \int_{-\infty}^x \e^{-\lam (x-y)}f(y) \ud y +  \tfrac 1{2\lam} \int_x^{\infty} \e^{-\lam (y-x)}f(y) \ud y, \qquad x \in \R. \]

Each cosine family generator is automatically the generator of a strongly continuous semigroup (but not vice versa). The semigroup  such operator  generates is given by the \emph{Weierstrass formula} (see e.g. \cite{abhn}*{p. 219})
\begin{equation}\label{cosf:1}T(0)f =f    \mqquad {and} T(t) f  = {\textstyle \frac 1{2\sqrt{\pi t}}} \int_{-\infty}^\infty \e^{-\frac {s^2}{4t}} C(s) f  \ud s, \qquad t >0, f \in \mathsf F.\end{equation} 
This formula expresses the fact that the cosine family is then, in a sense,  a~more fundamental object than the semigroup, and properties of the semigroup can be hidden in those of the cosine family.

Finally, a couple of words are in order about convergence of cosine families. Apparently, the theory is rather analogous to that of convergence of semigroups, see \cite{goldi,konishi,goldstein,laplace,knigazcup}: a sequence of equibounded cosine families with common bound converges if so does the sequence of their Laplace transforms. Nevertheless, caution needs to be exercised: 
even if convergence of, say, Feller semigroups is established, and we are able to show that the generators of these semigroups are also generators of cosine families, it is not immediate that these cosine families converge, even though convergence of Laplace transforms is granted. We still need to show that these families have a common bound. Sometimes it is not so easy, and sometimes the bound simply does not exist. See also \cite{zwojtkiem}.

\section{The main results} \label{sec:tmr}
Let $\ka\ge 2$ be a natural number, and let 
\[ \kad \coloneqq \{1,\dots, \ka \}.\]
We will be concerned with Feller processes that have values on the compactified infinite star graph \[ S\coloneqq K_{1,\ka}\]  with $\ka$ edges; this graph is obtained from $\ka$ copies of the compact half-line $[0,\infty]$, by lumping all zeros into a single vertex, denoted in what follows $\zero$,  see Figure \ref{rys1}. We will find it convenient to describe the space of continuous functions on $S$  in terms of the space $C[0,\infty]$, the latter being the space of continuous $f\colon [0,\infty) \to \R$ such that the limit $\lim_{x\to \infty} f(x)$ exists and is finite. We will namely identify $C(S)$ with the subspace of the Cartesian product $(C[0,\infty])^\ka$ of $\ka$ copies of $C[0,\infty]$ composed of vectors $\seq{f_i}$ such that $f_i(0)$ does not depend on $i\in \kad$. 

\subsection{Convergence of semigroups} 

Let us begin by introducing three main semigroups featured in our result: the semigroup of sticky Walsh process on a star graph, the semigroup of stopped Brownian motion on this graph, and the semigroup of Walsh process with additional semi-permeable membranes on the edges. 

As for the first of these, 
given nonnegative $\alpha_i, i \in \kad$ and  $\beta$ such that $\sui \alpha_i =1$ we define the (sticky) Walsh process on $S$ as the process related to the Feller semigroup generated by the following operator $\gen $ (see e.g. \cite{abap,onportenko,fromsotows,kostrykin2012,lejayskew}). 
 The domain $\dom{\gen}$ 
is composed of $\seq{f_i}\in \coss$ such that 
\begin{itemize} 
\item [(a) ] each $f_i$ is twice continuously differentiable (with one-sided derivatives at $0$); the limits $ \limx f_i'(x), \limx f_i''(x)$ exist and are finite;  
\item [(b) ] $f_i''(0)$ does not depend on $i$; this common value is denoted $f''(\zero)$;
\item [(c) ] $\sui \alpha_i f_i'(0) = \frac \beta 2 f''(\zero ) $.  
\end{itemize} 
Notably, even though $f_i'(0)$ depends on $i$, condition (b) implies that  $\seq{f''}$ is a member of $\coss$, and thus it makes sense to agree that 
\[ \gen \seq{f_i} \coloneqq  \pol \seq{f_i''}. \]

Next, we describe the semigroup of the \emph{stopped} Brownian motion on $S$, the process that stays at $\zero$ after reaching it for the first time. Its characteristic boundary condition is obtained (informally) from that featured in (c) by  letting $\beta \to \infty$. By that we mean that  the domain of its generator, say, $\abs$, is composed of $f = \seq{f_i}\in \coss $ such that 
\begin{itemize} 
\item [(a) ] each $f_i$ is twice continuously differentiable, 
\item [(b) ] the derivatives $f_i'$ and $f''_i$ have finite limits at $\infty$,   
\item [(c) ] $f''_i(0)=0, i \in \kad$.   
\end{itemize} 
As in the previous case, for $f\in \dom{\abs}$,  $\abs f = \pol \seq{f_i''}$.

Finally, 
to describe the approximating processes featured in our main theorem (Theorem \ref{main}) we additionally need a family of parameters $c_i, i \in \kad,$  specifying the locations of the soft membranes at the edges of the graph, together with parameters $p_{\eps,i}\in (0,1), \eps >0, i \in \kad,$ which characterize their permeability.
 We fix $\alpha_i,i \in \kad $ such that $\sui \alpha_i =1$, and $\beta\ge 0 $ and consider the operators $\gep$ defined as follows. For $\seq{f_i}$ in the domain of $\gep$,   
\begin{itemize} 
\item [(a) ] each $f_i$ is twice continuously differentiable in either of the sets $[0,c_i\eps]$  and $[c_i\eps, \infty) $ with one-sided derivatives at the ends, 
\item [(b) ] $ (1-\pepi) f'_i(c_i\eps - )=\pepi f_i'(c_i\eps+)$ whereas $f_i''(c_i\eps -)= f''_i(c_i \eps+), i \in \kad$, 
\item [(c) ] the limits $\limx f'_i(x)$ and $\limx f_i''(x)$ exist and are finite, 
\item [(d) ] $f_i''(0)$ does not depend on $i$; this common value is denoted $f''(\zero)$;
\item [(e) ] $\sui \alpha_i f_i'(0) = \frac \beta 2 f''(\zero ) $.  \end{itemize} 
As in the case of $\gen$, for $\seq{f_i}\in \dom{\gep}$ it makes sense to speak of $f_i''(c_i\eps)$ even though $f_i'(c_i\eps)$ in general does not exist, and $\seq{f_i''}$ is a member of $\coss$; we agree that 
\[ \gep \seq {f_i} \coloneqq \pol \seq{f_i''}. \]

Our main result asserts convergence of the semigroups generated by $\gep$ to the semigroup of Walsh's process with modified parameters. The function $s$ featured in the theorem measures the speed of convergence.

\begin{thm}\label{main} In the scenarios described below, 
\[ \grae \e^{t\gep} f  = \e^{t\genm} f, \qquad f\in \coss, t\ge 0, \] 
uniformly for $t$ in compact subintervals of $[0,\infty)$, where  $\genm$ is the generator of the Walsh process with modified parameters  $\wt \alpha_i, i \in \kad $ and $ \wt \beta$, prescribed  as follows. 
 \begin{itemize} \item [(a) ] If $\grae \pepi = p_i \in [0,1), i \in \kad$ with at least one $p_i\alpha_i$ nonzero, then 
  \[ \wt \alpha_i \coloneqq  \gamma^{-1} \tfrac {p_i\alpha_i}{1-p_i}, \quad i \in \kad   \mquad {and} \wt \beta \coloneqq  \gamma^{-1} \beta , \] where $\gamma \coloneqq \sum_{j\in \kad} \frac {p_j\alpha_j}{1-p_j}$ is the normalizing constant.  If, however, all $p_i\alpha_i$s vanish but $\beta >0$, then 
\[ \wt \alpha_i \coloneqq 0, \quad i \in \kad,      \mquad {and} \wt \beta \coloneqq  \infty , \]
 that is, the limit process is the stopped Brownian motion. 
 \item [(b) ] If $\beta =0$ and there is a function $s\colon (0,\infty) \to (0,\infty)$ such that
$\grae s(\eps)$ $=0$,  $\grae \frac \eps {s(\eps)} =0$, and $\grae \frac  \pepi {s(\eps)} = p_i \ge 0, i \in \kad$ with at least one $p_i\alpha_i$ nonzero, then  \[ \wt \alpha_i \coloneqq  \tfrac{p_i\alpha_i}{\sum_{j\in \kad} p_j \alpha_j}, \quad i \in \kad    \mquad {and} \wt \beta \coloneqq  0.  \]  
\item [(c) ] If $\beta =0$, and $\grae \frac {\pepi} \eps = p_i \ge 0, i \in \kad$ with at least one $p_i\alpha_i$ nonzero, then \[ \wt \alpha_i\coloneqq  \tfrac{p_i\alpha_i}{\sum_{j\in \kad} p_j \alpha_j}, \quad i \in \kad    \mquad {and}   \wt \beta \coloneqq \tfrac {2 \sui \alpha_i c_i  }{\sum_{j\in \kad} p_j \alpha_j}.  \] 
\item [(d) ]  If $s\colon (0,\infty) \to (0,\infty)$ is such that $\grae s(\eps)=0$, $\grae \frac \eps {s(\eps)} =0$, and $\grae \frac {1-\pepi}{s (\eps)} = q_i>0 $, then 
  \[ \wt \alpha_i \coloneqq  \gamma^{-1} \tfrac {\alpha_i}{q_i}, \quad i \in \kad   \mquad {and} \wt \beta \coloneqq  0 , \] where $\gamma\coloneqq \sum_{j\in \kad} \frac {\alpha_j}{q_j}$ is the normalizing constant.  
  \end{itemize}
\end{thm}

By the Trotter--Kato--Neveu--Macke\-vi\v cius Theorem, our result can be equivalently stated by saying that in the scenarios described above the processes $X_\eps$ underlying $\gep$, seen as random elements with values in  $D([0,\infty), S)$,  converge weakly in  distribution to the Walsh process, provided that the initial distributions of these processes converge weakly also.

Notably, as (a) reveals in particular,  if some of the limit probabilities $p_i$ are positive, but the remaining $p_i$s vanish, the limit process cuts off the possibility of entering the edges corresponding to the latter $p_i$s, even if the original $\alpha_i$ was nonzero.  Also, by (b), a similar effect is observed if the rates of convergence of $\pepi, i \in \kad$ to $0$ differ. Interestingly, furthermore, if $\pepi$s converge to $0$ rapidly (see (c)), the membrane at $\zero$ becomes sticky, even if the approximating ones are not.  Case (a) also reveals that the limit process may have a trap at $\zero$ if the approximating ones have even a small amount of stickiness at $\zero$.  By (d), if all the approximating membranes, with similar but not too  high a rate, have a tendency to reflect the paths away from $\zero,$  the limit process may lose its stickiness. Finally, it is remarkable that  parameters $c_i, i =1,\dots, \ka$ do not influence the limit in the cases (a), (b) and (d) at all.

\subsection{Convergence of cosine families}

Our main result (Thm. \ref{main}) is accompanied by the information on the convergence of the corresponding cosine families. Here, because of the limitations of the Lord Kelvin method of images involved in our analysis, we restrict ourselves to the case where 
\[ \ka =2, \beta =0, \quad \alpha_1=\alpha_2 = \pol   \quad \text{ and } \quad c_1=c_2=1,\]
and focus on the scenario described in the first part of Theorem \ref{main} (a). This allows us to dive deeper into the structure of the objects under study. The related process can be thought of as a Brownian motion skewed at two points: $-\eps$ and $\eps$.

Here is the main result on cosine families. 

\begin{thm}\label{dodatek} Let $\gep$ be the operator defined before Theorem \ref{main}.
\begin{itemize} \item [(a) ] Each $\gep, \eps >0$ is  the generator of a cosine family $C_\eps$ such that 
\begin{equation}\label{cocf:1} \|C_\eps(t) \| \le 11 + \tfrac {12}{1-|q_\eps|}, \qquad t \in \R,  \end{equation}
where 
 \begin{equation}\label{cocf:0} q_\eps  \coloneqq (1-2p_{\eps,1}) (1-2p_{\eps,2}) \in (-1,1). \end{equation}

\item [(b) ] If $\grae p_{\eps,i} = p_i \in [0,1)$, and $q\coloneqq (1-2p_1)(1-2p_2)$ differs from $1$ (equivalently, $ p_{1}+p_{2} >0$), then 
\[ \grae \sup_{t\in \R} \|C_\eps (t) f  - C_0(t)f \| = 0, \qquad f \in \coss , \]
where $C_0$ is the cosine family generated by $\mathfrak H_{\wt \alpha,0} $ with (cf. Thm. \ref{main} (a)) 
\begin{equation} \label{wojoj} \wt \alpha_i \coloneqq \gamma^{-1} \tfrac {p_i}{1-p_i}, i =1,2 \quad \text{and} \quad \gamma \coloneqq \tfrac {p_1}{1-p_1}
+  \tfrac {p_2}{1-p_2}.\end{equation}
\end{itemize}
\end{thm}

\begin{remark}

Thm. \ref{dodatek} (b), when combined with the  Weierstrass formula \eqref{cosf:1}, implies that (see e.g. \cite{abernachr}*{Remark 2.1})
\[ \grae \sup_{t\ge 0} \|\e^{t\gep }f - \e^{t \mathfrak H_{\wt \alpha,0}}f \| =0, \qquad f \in \coss\] 
and thus strengthens the first part of Thm. \ref{main} (a).  \end{remark}

 \begin{remark} 

The cosine family generated by   $\mathfrak H_{\wt \alpha,0} $
can be written explicitly as (see  \cite{abtk}*{Section 6} or \cite{fromsotows}*{Thm. 5.1 (c)})
\begin{equation}\label{dodatek:8} (C_0 (t)f)_i (x) =   C\big (\textstyle{\frac {\sqrt 2}2 t}\big)\fir(x), \quad  f=(f_1,f_2) \in \coss,x\ge 0, t \in \R,  
\end{equation} 
where $\{C(t), t\in \R\}$ is the basic cosine family defined in Section \ref{sec:cf} and 
\begin{align} 
 \fir (x) & \coloneqq \begin{cases} f_i(x) & x\ge 0, \\
(2\wt \alpha_i-1) f_i(-x) + 2\wt \alpha_j f_j(-x), & x<0, \end{cases}\label{ext_skew} \end{align}
  \( i =1,2, j=3-i. \)
 \end{remark}

\section{Proof of Theorem \ref{main}}

\subsection{The resolvent of $\gen$}\label{sec:tro}
We begin by noting that the resolvent of the Walsh process can be given explicitly (see e.g. \cite{abap}) in terms of the resolvent $\rla^0, \lam >0$  of the minimal Brownian motion on $S$, the latter being defined, for all $g=\seq{g_i} \in \coss, \lam >0$ and $x\ge 0$,  by \begin{align} 
(\rla^0 g)_i(x) 
&= C_i(g) \sinh \slam x - \sqrt {\tfrac 2{\lambda}}\int_0^x \sinh \slam (x-y) g_i(y)\ud y, \label{tro:1} \end{align}
where 
\begin{align*}C_i(g)&  \coloneqq \slem \czn \e^{-\slam y} g_i(y)\ud y, \qquad i \in \kad.   \end{align*}
To wit, we have 
\begin{equation}\label{tro:2} \rez{\gen} g = \rla^0 g + E(g) \elam , \qquad g\in \coss, \lam >0, \end{equation}  
where 
\begin{align}\label{tro:3a} 
E(g)& \coloneqq \frac {\slam \sui \alpha_i C_i(g) + \beta g(\zero)}{\lam (\slam \sui \alpha_i C_i(\jeden_S)+ \beta)} 
, \end{align}
and 
\begin{equation*} (\elam )_i (x) \coloneqq \e^{-\slam x}, \qquad x \ge 0, i \in \kad. \end{equation*}
The resolvent of the stopped process has a similar form: we have 
\[ \rez{\abs} g  = \rla^0 g + \tfrac{g(\zero)}\lam \elam, \qquad g \in \coss, \lam >0.  \]

\newcommand{\cepi}{C_{\eps,i}}
\newcommand{\depi}{D_{\eps,i}}

\subsection{A counterpart of $\rla^0, \lam >0$} % and the resolvent of $\gep$}

The basic idea of the proof of Theorem \ref{main} is to represent the resolvent of each $\gep$ in a form resembling \eqref{tro:2}. To this end, we fix $\eps >0$ and introduce the family $\rlae^0, \lam >0$, which  for $\rez{\gep}$  plays a role analogous to that played by $\rla^0, \lam >0$ for $\rez{\gen}, \lam >0$. This family describes the process, say, 
\[ X_\eps^0,\]  which, while on the $i$th edge, $i \in \kad$, behaves as a Brownian motion that passes through a~soft membrane located at a distance $c_i \eps >0$ from the graph center~$\zero$, but is undefined from the first moment it touches $\zero$.

We start with the definition of the related operator $\gepo$. Its domain $\dom{\gepo}$ is composed of  $\seq{f_i}\in \coss$ such that conditions (a)--(d) of the definition of $\dom{\gep}$ are satisfied but (e) is replaced by: 
\begin{itemize}
\item [(e') ]  $f_i (0)=0, i \in \kad$.   
\end{itemize} 
As above, even though $f_i'(c_i \eps)$ in general does not exist, condition (b) implies that  it makes sense to speak of $f_i''(c_i \eps), i \in \kad$. Also, by (c) and (d), $\seq{f''}$ is a member of $\coss$, and thus we can agree that 
\[ \gepo \seq{f_i} \coloneqq  \pol \seq{f_i''}. \] 
It is probably worth stressing here that $\gepo$ is not densely defined, but otherwise enjoys all the properties of a Feller generator, and thus in particular falls into the category studied in the seminal, by now classical paper \cite{dapratos}.
The following lemma contains information on the resolvent of $\gepo$.

\begin{lemma}\label{lemik:1}Let $g=\seq{g_i}\in \coss$, $\eps,\lam>0$ and $ i\in \kad $ be fixed. Also, let 
\begin{align}\nonumber f_i(x) & \coloneqq \Big (\cepi (g) \sinh \slam x - \slem \int_0^x \sinh \slam (x-y)g_i(y)\ud y \Big ) [0\le x < c_i \eps] 
\\&\phantom{=} + \Big (\depi (g) \e^{-\slam (x-c_i \eps )} + \tfrac 1\slam \int_{c_i \eps}^\infty \e^{-\slam |x-y|}g_i(y)\ud y \Big)[x\ge c_i \eps], \label{tro:3} \end{align}
where $[ \cdot]$ is the Iverson bracket, 
\begin{align*}%\label{20kw:2} 
\cepi (g) &\coloneqq  \frac {\pepi \slem \int_0^\infty \e^{-\slam y}g_i(y+c_i \eps)\ud y + \slem \int_0^{c_i \eps} \chi_{\eps,i} (c_i \eps - y)g_i(y)\ud y}{\chi_{\eps,i} (c_i \eps) },\\
\depi (g) & \coloneqq \cepi (g)  \sinh \slam c_i \eps - \slem \int_0^{c_i \eps} \sinh \slam (c_i \eps - y)g_i(y)\ud y \\ &\phantom{=}-\tfrac 1\slam \int_0^\infty \e^{-\slam y}g_i(y+c_i \eps)\ud y, \end{align*}
and, to simplify the formula, $\chi_{\eps,i} (x) \coloneqq  (1-\pepi )\cosh \slam x + \pepi \sinh \slam x,$  $x \in \R$. Then, $f\coloneqq \seq{f_i}$ is a unique solution to the resolvent equation $\lam f - \gepo f = g$, and so in what follows we will write  
\[f =\rlae^0 g \coloneqq  \rez{\gepo}g. \]   
\end{lemma}

\begin{proof} First, we need to check that $f$ belongs to $\dom{\gepo}$. To begin with, continuity of $f_{i}$ can be an issue only at  $x=c_i \eps$, but here it is guaranteed by the definition of $\depi$. Moreover, a direct calculation shows that each $f_{i}$ is twice continuously differentiable on $[0,c_i \eps]$ and $[c_i \eps, \infty)$ separately  (with one-sided derivatives at the ends), and in each of these intervals solves $\lam f_i - \pol f''_{i} = g_i $. This, in particular, takes care of part (a) in the definition of $\dom{\gepo}$.

Turning to (b), we note that the first  condition featured in it is a result of the special choice of $\cepi$ and $\depi$. Also, the second condition follows by the already established differential equation and the fact that $g_i$ and $f_{i}$ are continuous.

Next, the limits $\limx f_{i}(x)$ and  $\limx f'_{i}(x)$ obviously exist and are finite, $i\in \kad$.  The same statement on $\limx f_i''(x)$ is a direct consequence of the already established differential equation and the fact that $\limx f_{i} (x)$ and $\limx g_i(x)$ exist and are finite too. This shows (c).

Finally, (e') is obvious, and the  differential equation establishes also that $f''_i(0)$ does not depend on $i$, because neither do $g_i(0)$ and $f_i(0)$, thus proving (d). Now, once the fact that $f$ belongs to $\dom{\gepo}$ is proven, the differential equation turns out to be just another form of $\lam f - \gepo f = g$.  

Uniqueness of solution is a consequence of the fact that, as can be easily seen, $\gepo$ satisfies the positive-maximum principle, and thus is dissipative (see e.g. \cite{ethier}*{Lemma 2.1, p. 165}).
\end{proof}

 Before proceeding we note that  
\[ \cepi (\jeden_S) = \frac {(1-\pepi)\sinh \slam c_i \eps + \pepi \cosh \slam c_i \eps}{\lam ((1-\pepi)\cosh \slam c_i \eps + \pepi \sinh \slam c_i \eps)}. \] 
Consequently,  
 \[ \elame \coloneqq \jeden_S - \lam \rlae^0  \jeden_S \] can be calculated to be given by  
\begin{align*} (\elame)_i (x) &= (\cosh \slam x-  \cepi (\jeden_S) \lam \sinh \slam x )[0\le x< c_i \eps ] \\&\phantom{=}+ \frac {(1-\pepi )\e^{-\slam (x-c_i \eps)}}{(1-\pepi ) \cosh \slam c_i \eps + \pepi \sinh \slam c_i \eps}[x\ge c_i \eps]. \end{align*}
The function $\lam \mapsto (\elame)_i (x) $ is the Laplace transform of the life-time of the underlying process $X_\eps^0$ starting at the $i$th edge at the distance $x\ge 0$ from the graph center.  Notably, but not surprisingly, the second term above is the product of two Laplace transforms: the first of these, $\lam \mapsto \e^{-\slam (x-c_i \eps)}$   is the transform of the time needed for a one-dimensional Brownian motion started at the $x>c_i \eps$ to reach $c_i \eps $ for the first time, and the second, $\lam \mapsto  \tfrac {1 - \pepi}{(1-\pepi)\cosh \slam c_i \eps + \pepi \sinh \slam c_i \eps}$ is the life-time of $X_\eps^0$  started at the $i$th edge at  the distance $c_i \eps$ from the graph center.  

\subsection{The resolvent of $\gep$} 

With $\rlae^0, \lam >0$ under our belt, we define the family $\rlae , \lam >0$ by 
\begin{equation}\label{tro:4} \rlae g \coloneqq  \rlae^0 g + E_{\eps} (g) \elame , \qquad g\in \coss, \lam >0, \end{equation}  
where $\elame $ was described above and 
\begin{align*}
E_\eps (g)& \coloneqq \frac {\slam \sui \alpha_i \cepi (g) + \beta g(\zero)}{\lam (\slam \sui \alpha_i \cepi (\jeden_S)+ \beta)}. \end{align*}

\begin{proposition}\label{lemik:2} Each $\gep$ is a Feller generator, and $\rlae, \lam >0$ introduced in \eqref{tro:4} is its resolvent. \end{proposition}

\begin{proof} It is clear that $\gep$ is densely defined and a short argument establishes that it satisfies the positive-maximum principle. Hence, it remains to check that, for every $g\in \coss$ and $\lam >0$, $f\coloneqq \rlae g$ solves the resolvent equation $\lam f - \gep f = g$.

Let $\Delta_\eps$ be the extension of the operator $\gep$ to the domain where conditions (a)--(d) of the definition of $\dom{\gep}$ are satisfied, but (e) is possibly violated. Clearly, $\Delta_\eps$ extends $\gepo$ as well. Also, let $F\colon \dom{\Delta_\eps} \to \R$ be given by 
\[ F \seq{f_i} = \sui \alpha_i f_i'(0) - \tfrac \beta 2 f''(0),\]
so that $\seq{f_i}\in \dom{\Delta_\eps}$ belongs to $\dom{\gep}$ iff $F\seq{f_i}=0$.

 By Lemma \ref{lemik:1}, for each $g\in \coss$,  $\rlae^0g $ belongs to $\dom{\gepo}\subset \dom{\Delta_\eps}$ and we have $\Delta_\eps  \rlae^0 g = \lam \rlae^0 g  - g$.  Since $\jeden_S$ belongs to $\dom{\Delta_\eps}$, it follows that so does $\elame $, and we have  $\Delta_\eps \elame = - \lam \Delta_\eps  \rlae^0 \jeden_S = \lam (\jeden_S - \lam \rlae^0 \jeden_S) = \lam \elame$, and this yields 
 \[ \lam \rlae g - \Delta_\eps \rlae g = g .
\] 
Moreover, a direct computation shows that $F\rlae^0 g = \slam \sui \alpha_i \cepi (g) + \beta g(\zero)$, and in particular $F\elame = -\lam (\slam \sui \alpha_i \cepi (\jeden_S) + \beta)$. Hence, $E_\eps (g)$ is so chosen that $F\rlae g =0$, and the display above turns into $ \lam \rlae g - \gep \rlae g = g$. This completes the proof.  
 \end{proof}

\subsection{Limit behavior} 

The following lemma describes the four basic scenarios of asymptotic behavior of functionals $\cepi $ and $E_\eps$ introduced in \eqref{tro:3} and \eqref{tro:4}, respectively. 

\begin{lemma}\label{lemik:3}  \ \ 

 \begin{itemize} \item [(a) ] If $\grae \pepi  = p_i \in [0,1)$ then $\grae \cepi (g) = \frac {p_i}{1-p_i} C_i (g), g\in \coss ,$  $i \in \kad $, for $C_i(g)$s introduced in \eqref{tro:1}.  As a result, as long as $\sui p_i \alpha_i +\beta $ is nonzero (so that the denominator below is nonzero also),
\[ \grae E_\eps (g) = E(g) \coloneqq \frac {\slam \sui \frac{p_i\alpha_i}{1-p_i} C_i (g) + \beta g(\zero)}{\lam (\slam \sui \frac{p_i\alpha_i}{1-p_i} C_i (\jeden_S)+ \beta)}. \]

 \item [(b) ] If $\beta =0$ and $s\colon (0,\infty) \to (0,\infty)$ is such that $\grae s(\eps)=\grae \frac \eps {s(\eps)} =0$, and $\grae \frac  \pepi {s(\eps)} = p_i \ge 0$ then $\grae \frac {\cepi (g)}{s(\eps)} = p_i C_i (g), g\in \coss,  i \in \kad  $. As a result, 
    as long as $\sui \alpha_i p_i $ is nonzero,
\[ \grae E_\eps (g) = E(g) \coloneqq  \frac {\sui {p_i\alpha_i} C_i (g)}{\lam \sui {p_i\alpha_i} C_i (\jeden_S)} .\]

\item [(c) ] If $\beta =0$ and  $\grae \frac {\pepi} \eps = p_i\ge 0 $, then  $\grae \frac{\cepi (g)}\eps = p_i C_i(g) + c_i\slem g(\zero), g\in \coss, i \in \kad$. As a result, 
\[ \grae E_\eps (g)  = E(g) \coloneqq  \frac {\slam \sui {p_i\alpha_i} C_i (g) + 2g(\zero) \sum_{j\in \kad} \alpha_j c_j }{\lam (\slam \sui {p_i\alpha_i} C_i (\jeden_S)+ 2 \sum_{j \in \kad } \alpha_j c_j )}. \]
\item [(d) ] If $s\colon (0,\infty) \to (0,\infty)$ is such that $\grae s(\eps)=\grae \frac \eps {s(\eps)} =0$, and $\grae \frac {1-\pepi}{s (\eps)} = q_i>0 $, then  $\grae s(\eps) \cepi (g) = \frac{C_i(g)}{q_i}, g\in \coss, i \in \kad$. As a result, 
\[ \grae E_\eps (g)  = E(g) \coloneqq  \frac { \sui \frac{\alpha_i}{q_i} C_i (g)}{\lam \sui \frac{\alpha_i}{q_i} C_i (\jeden_S)}. \]

  \end{itemize}
\end{lemma}

Omitting an elementary proof of the lemma, we turn to our next auxiliary  result which settles the question of the convergence of the resolvents  $\rlae^0, \lam >0$ as $\eps \to 0$.

\begin{lemma}\label{lemik:4}Suppose that $\pepi$s are chosen so that $\grae \eps \cepi (g)=0$ for all $i\in \kad$ and $g\in \coss.$ Then, $\grae \rlae^0 = \rla^0, \lam >0$ in the strong topology, where $\rla^0, \lam >0 $ is the resolvent of the minimal Brownian motion introduced in \eqref{tro:1}. Otherwise, convergence fails to hold.  \end{lemma}

\begin{proof}We begin by noting that in the case where all $\pepi$s are equal $\pol$,  the last two conditions in point (b) of the definition of $\dom{\gepo}$ simply say that $f_i$s are to be twice continuously differentiable at $x=c_i\eps$ also. Then, $X_\eps^0$ is nothing else but the minimal Brownian motion. This means that representation \eqref{tro:3} applies also to $\rla^0, \lam >0$, and a short calculation reveals that in this case $\cepi$ coincides with $C_i$ introduced in \eqref{tro:1}. Consequently, by the definition,  $\depi (g)$ takes the form $\frac 1\slam \int_0^{c_i\eps} \e^{\slam (y-c_i\eps)} g_i(y)\ud y - \frac 1\slam \int_0^\infty \e^{-\slam (y+c_i\eps)} g_i(y)\ud y, g \in C(S)$; to distinguish it from the general $\depi (g)$, we will denote this value by $\overline {\depi} (g)$. Clearly, $\grae \overline{\depi} (g)= - \pol C_i(g)$.  

But, by assumption and the definition of $\depi$, $\grae \depi(g) = - \pol C_i$ also. It follows that, for each $i\in \kad$, 
\begin{equation*}%\label{lb:1} 
\sup_{x\ge c_i\eps}\big |(\rlae^0 g)_i(x) - (\rla^0 g)_i(x)\big| = |\depi (g) - \overline{\depi}(g)| \underset {\eps \to 0} \longrightarrow  0. \end{equation*}
Moreover, again by assumption, 
\begin{align*} \max_{i \in \kad} \sup_{x\in [0,c_i\eps]} & \big |(\rlae^0 g)_i(x)  \big |\\  &\le  \cepi (g) \sinh \slam c_i \eps  
 +\slem \|g\|   \int_0^{c_i\eps} \sinh \slam y \ud y \underset {\eps \to 0} \longrightarrow  0.\end{align*}
Since, as $\eps \to 0$,   $\max_{i \in \kad} \sup_{x\in [0,c_i\eps]} |(\rla^0 g)_i(x)| $ also converges to $0$ (because $\rla^0 g$ is a continuous function on $S$ which vanishes at $\zero$), 
the last two displays complete the proof of the first part of the lemma.

Now, if condition $\grae \eps \cepi(g) =0$ does not hold for some $i\in \kad, \lam >0$ and $g\in \coss$, there is a $\delta >0 $ and $\jcg{\eps_n}$ such that $\gra \eps_n = 0$ and $|C_{\eps_n,i} (g) \sinh \slam c_i \eps_n   |\ge \delta, n \ge 1$. Hence, $\limsup_{n\to \infty} |D_{\eps_n,i} (g)- \overline{D_{\eps_n,i} }(g)| \ge \delta $.  This, however, by 
\( \|\rlae^0 g - \rla^0 g \| \ge  |(\rlae^0 g)_i (c_i\eps) - (\rla^0 g)_i(c_i\eps)| = |\depi (g)- \overline {\depi}(g)|\)
excludes the possibility of having $\grae \|\rlae^0 g - \rla^0 g \|=0$. 
 \end{proof}

We note in passing that, if the existence of $p_i \coloneqq \grae p_{\eps,i}, i \in \kad $ is granted,  condition  $\grae \eps \cepi (g)=0$ for all $i\in \kad$ and $g\in \coss$ may be violated only if at least one $p_i $ is $1$ --- see Lemma \ref{lemik:3} (a).

\subsection{Proof of Theorem \ref{main}}

We are finally in a position to prove Theorem \ref{main}. We begin by noting that all four scenarios of the theorem correspond to these of Lemma \ref{lemik:3}. 
In particular, the lemma shows that in all the scenarios $\grae \eps \cepi (g) =0, g \in \coss, i \in \kad$. Hence, by Lemma \ref{lemik:4}, $\grae \rlae^0 = \rla^0$ strongly, and this implies, by the definition of $\elame$, that $\grae \elame = \elam $, for all $\lam >0$.  Since Lemma \ref{lemik:3} establishes also the convergence of $E_\eps (g)$ to $E(g)$, \eqref{tro:4} tells us that $\grae \rlae g = \rla^0g + E(g) \elam, g \in \coss, \lam >0$. Therefore, by \eqref{tro:2} and the  Trotter--Kato--Neveu--Macke\-vi\v cius theorem, it remains to recognize, in each case separately, $E(g)$ as that corresponding to the resolvent of the appropriate $\genm$, as spelled out in \eqref{tro:3a}.

  For example, in the case (a), if at least one $p_i\alpha_i$ is positive,  we can divide the numerator and the denominator of $E(g)$ by the normalizing constant $\gamma \coloneqq \sum_{j\in \kad} \frac {p_j\alpha_j}{1-p_j}$ to see that the $E(g)$ characterizes $\genm$ with 
\( \wt \alpha_i \coloneqq  \gamma^{-1} \tfrac {p_i\alpha_i}{1-p_i}\) and  \(\wt \beta \coloneqq  \gamma^{-1} \beta\). However, if all $p_i\alpha_i$s vanish, the formula for $E(g)$ reduces to $E(g) = \frac {g(\zero)}\lam$, which we recognize as that corresponding to the stopped process. We omit the discussion of the remaining cases.

\section{Proof of Theorem \ref{dodatek}}

\newcommand{\fil}{ f_i^{\mathpzc l}}
\newcommand{\fjr}{ f_{3-i}^{\mathpzc r}}
\newcommand{\fjl}{ f_{3-i}^{\mathpzc l}}
\newcommand{\gir}{ g_i^{\mathpzc r}}
\newcommand{\gilp}{ g_{i,+}^{\mathpzc l}}
\newcommand{\gilm}{ g_{i,-}^{\mathpzc l}}
\newcommand{\gjlp}{ g_{j,+}^{\mathpzc l}}
\newcommand{\gjlm}{ g_{j,-}^{\mathpzc l}}
\newcommand{\hir}{ h_i^{\mathpzc r}}
\newcommand{\hil}{ h_i^{\mathpzc l}}
\newcommand{\fire}{ f_{\eps,i}^{\mathpzc r}}
\newcommand{\file}{ f_{\eps,i}^{\mathpzc l}}
\newcommand{\gire}{ g_{\eps,i}^{\mathpzc r}}
\newcommand{\gilpe}{ g_{\eps,i,+}^{\mathpzc l}}
\newcommand{\gilme}{ g_{\eps,i,-}^{\mathpzc l}}
\newcommand{\girz}{ g_{0,i}^{\mathpzc r}}
\newcommand{\gilpz}{ g_{0,i,+}^{\mathpzc l}}
\newcommand{\gilmz}{ g_{0,i,-}^{\mathpzc l}}
\newcommand{\firz}{ f_{0,i}^{\mathpzc r}}
\newcommand{\filz}{ f_{0,i}^{\mathpzc l}}

\subsection{Proof of part (a)} \label{sec:popa}

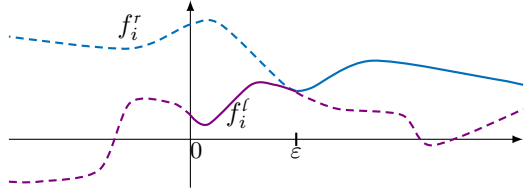
\begin{figure}
 \begin{tikzpicture}[scale=0.8]
 \draw [->] (-3,0) -- (5.5,0);
\draw [->] (0,-0.8)-- (0,2.3);
\draw [violet, thick, densely dashed] plot [smooth] coordinates {(-3,-0.7)(-2.5,-0.7) (-1.5,-0.5)(-1,0.6) (-0.3,0.6) (0,0.4)};
\draw [violet, thick,  densely dashed]  plot [smooth] coordinates {(1.75,0.78) (1.9,0.7)  (2.45,0.5)  (3.5,0.4) (4,-0.1)(5.5,0.5)};
\draw [NavyBlue, thick,  densely dashed]  plot [smooth] coordinates {(1.75,0.8)(1.5,0.95)  (0.5,1.9) (0,1.9)(-1,1.5)(-3,1.7)};
\draw [NavyBlue,thick]  plot [smooth] coordinates {(1.75,0.8) (2,0.85)  (3,1.3) (5,1)(5.5,0.9)};
\draw [violet,thick]  plot [smooth] coordinates {(0,0.4)(0.3,0.25)  (1,0.9)(1.4,0.9)(1.75,0.79)};
\node [above] at (0.8,0.1) {$\fil$};
\node [above] at (-1,1.55) {$\fir$};
\draw [ thick,-] (1.75,-0.11) -- (1.75,0.11);
\node [above] at  (0.1,-0.4) {\small{{$0$}}}; 
\node [above] at  (1.75,-0.4) {\small{{$\varepsilon$}}}; 
\end{tikzpicture}

\caption{\footnotesize{Two extensions of $f_i \in C[0,\infty]$ (solid line): the violet curve is the extension of its left part, and the blue curve is the extension of its right part. Notably, as spelled out in \eqref{dodatek:1}, these extensions depend on both $f_1$ and $f_2$. }}\label{rys3} 
\end{figure}

Our argument is based on the Lord Kelvin method of images \cites{kosinusy,tombatty,bobrat,aberman}, which provides semi-explicit formulae for the cosine families involved (cf. our formula \eqref{dodatek:8}). The method hinges on constructing very particular extensions of members $(f_1,f_2)$ of $\coss$, extensions that are shaped by the boundary conditions. To elaborate, we think of each $f_i$ as composed of two parts, left and right, the left being the restriction of $f_i$ to the interval $[0,\eps]$ and the right being the extension of $f_i$ to $[\eps, \infty)$, and search for continuous extensions of these two parts to the entire real line (see Figure \ref{rys3}).  All in all, given $(f_1,f_2)\in \coss $ we search for four functions, 
\[\fil ,\fir, \qquad i=1,2,\]
on the entire line, where  the superscript $\mathpzc l$ stands for `left' and $\mathpzc r$ for `right', and $\fil$ is the extension of the left part of $f_i$ whereas $\fir$ is the extension of the right part of $f_i$. 
We require that these functions belong to $C[-\infty, \infty]$, the space of continuous functions on $\R$ that possess finite limits at $\pm \infty$, and that the following conditions are satisfied: 
\begin{align} 
 \fil (-x) &= \fjl (x),  \label{dodatek:1} \\
 \pepi \fir (\eps - x) & = \pepi f_i (x+\eps) + (1-\pepi) \fil (\eps - x)  - (1-\pepi) \fil (x+\eps), \nonumber \\
 \fir (\eps - x) &= -f_i (x+\eps) + \fil (\eps - x) + \fil (x+\eps), \qquad i =1,2, x\ge 0. \nonumber
\end{align}
Here, and throughout the proof of part (a), $\eps >0$ is fixed and $\pepi$s are parameters characterizing the domain of $\gep$. 
%(To explain: the first equation above means that $f_1^{\mathpzc l} (-x) = f_2^{\mathpzc l} (x) $ and  $f_2^{\mathpzc l} (-x)$ $ = f_1^{\mathpzc l} (x) $ for all $x\ge 0$. We will use such shortened notation in what follows also.)  

\begin{figure}[t!]
\begin{tikzpicture}[scale=0.75]
\draw[->] (-3,0) -- (5.5,0);
\draw[->] (0,-0.5) -- (0,2.5);
\draw [ thick,-] (1.75,-0.11) -- (1.75,0.11);
\node [above] at  (0.1,-0.4) {\small{{$0$}}}; 
\node [above] at  (1.75,-0.4) {\small{{$\varepsilon$}}}; 
\draw [NavyBlue, thick, densely dashed]  plot [smooth] coordinates {(1.75,0.8)(1.5,0.95)  (0.5,1.9) (0,1.9)(-1,1.5)(-3,1.7)};
\draw [Green, thick, densely  dashed]  plot [smooth] coordinates {(1.75-1.75,0.8)(2-1.75,0.95)  (3-1.75,1.9) (3.5-1.75,1.9)(4.5-1.75,1.5)(6.5-1.75,1.7)(5.5,1.75)}; 
\draw [NavyBlue,thick]  plot [smooth] coordinates {(1.75,0.8) (2,0.85)  (3,1.3) (5,1)(5.5,0.9)};
%\draw[gray] (1.75,-0.5)--(1.75,2.5);
\draw[gray,->](0.25,2.1) .. controls (0.65,2.5) and (1.05,2.5) ..(1.45,2.1);
\node [above] at (-1,1.55) {$\fir$};
\node [above] at (2.2,1.8) {$\gir$};
\end{tikzpicture}\caption{\footnotesize{The function $\gir$ (green curve) is constructed from a portion of the function $\fir$ (blue curve),  by first shifting the latter to the left by $\eps$ and then reflecting it about the $y$ axis. Vice versa, the knowledge of $\gir$ is the knowledge of $\fir$, because the right side of $\fir $ is given.}}\label{rys_r}
\end{figure}
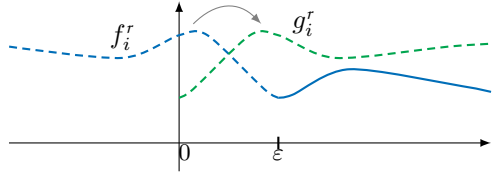

In Lemma \ref{lem:d1}, further down, we will establish that such extensions exist and are uniquely determined. Before we do that, however, we need to discuss the operation of concatenating graphs of functions. Given $a >0$ and a possibly infinite $b>0$, and two  functions, say, $k\in C[0,a]$ and $g \in C[0,b]$, such that $g(0)= k(a)$ we can consider concatenation of $k$ and $g$, denoted simply $kg$, the member of $C[0,a+b]$ given by 
\[ kg (x) = k(x) [ 0\le x\le a] + g(x-a) [a<x\le a+b]. \] 
 With fixed $k$, the map $g \mapsto kg$ is nonlinear, and defined only for $g$ such that $g(0) =k(a)$, but --- remarkably --- has not fallen far from the tree of linear operators. To wit, if $g_1\in C[0,b]$ and $g_2\in C[0,b]$ are such that $g_1(0)=g_2(0)= k(a)$ and numbers $\alpha$ and $\beta$ add up to $1$, we have 
\begin{equation}\label{dodatek:2} k (\alpha g_1 + \beta g_2) = (\alpha k)(\alpha g_1) + (\beta k)(\beta g_2) = \alpha (kg_1) + \beta (kg_2). \end{equation}
We will also use the fact that, as it is easy to see, 
\begin{equation}\label{dodatek:3} \|kg\|_{C[0,a+b]} = \max (\|k\|_{C[0,a]}, \|g\|_{C[0,b]}). \end{equation}

 \begin{lemma}\label{lem:d1} For any $f= (f_1,f_2)\in \coss$,  extensions $\fir$ and $\fil, i =1,2$ described above exist and are uniquely determined. Moreover, 
\begin{equation} \label{dodatek:3a} \|\fil\|_{C[-\infty,\infty]}, \|\fir\|_{C[-\infty,\infty]}  \le (11 + \tfrac {12}{1-|q_\eps|}) \| f \|_{C(S)}
%\max_{i=1,2} \|g_i\|_{C[0,\infty]}
, \quad  i=1,2.  \end{equation}

  \end{lemma}

\begin{proof}Instead of searching for $\fir$s and $\fil$s it will be convenient to search for the functions defined on $[0,\infty)$ as follows (see Figures \ref{rys_r} and \ref{rys_l}): %\ab{Another good figure is needed.} 
\begin{align*}
\gir (x) & \coloneqq \fir (\eps - x), \\
\gilp (x) & \coloneqq \fil (\eps + x), \\
\gilm (x) & \coloneqq \fil (- x), \qquad x\ge 0, i=1,2. 
\end{align*}
 We have thus produced six unknowns out of four, but the advantage is that they are all defined on the same domain. We note that $\gilp$ is the right side of the unknown part of $\fil$ and $\gilm$ is the left side of the unknown part of this function.

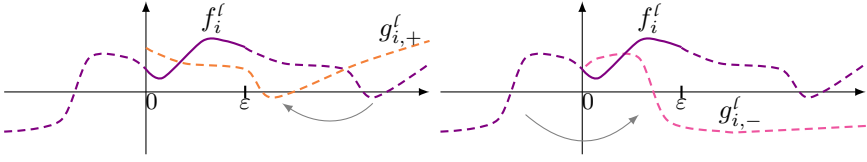
\begin{figure}[t!]
 \begin{tikzpicture}[scale=0.75]
 \draw [->] (-2.5,0) -- (5,0);
\draw [->] (0,-1.1)-- (0,1.6);
\draw [violet, thick, densely dashed] plot [smooth] coordinates {(-2.5,-0.7) (-1.5,-0.5)(-1,0.6) (-0.3,0.6) (0,0.4)};
\draw [violet, thick, densely dashed]  plot [smooth] coordinates {(1.75,0.78) (1.9,0.7)  (2.45,0.5)  (3.5,0.4) (4,-0.1)(5,0.5)};
\draw [Orange, thick, densely dashed]  plot [smooth] coordinates {(1.75-1.75,0.78) (1.9-1.75,0.7)  (2.45-1.75,0.5)  (3.5-1.75,0.4) (4-1.75,-0.1)(5.5-1.75,0.5)(7.25-2.25,0.9)};
\draw [violet,thick]  plot [smooth] coordinates {(0,0.4)(0.3,0.25)  (1,0.9)(1.4,0.9)(1.75,0.79)};
\node [above] at (1.2,0.95) {$\fil$};
\node [above] at (4.5,0.7) {$\gilp$};
\draw [ thick,-] (1.75,-0.11) -- (1.75,0.11);
\node [above] at  (0.1,-0.4) {\small{{$0$}}}; 
\node [above] at  (1.75,-0.4) {\small{{$\varepsilon$}}}; 
\draw[gray,->] (4,-0.2).. controls (3.6,-0.6)and  (3,-0.6) ..(2.4,-0.2);
\end{tikzpicture}
 \begin{tikzpicture}[scale=0.75]
 \draw [->] (-2.5,0) -- (5,0);
\draw [->] (0,-1.1)-- (0,1.6);
\draw [violet, thick, densely dashed] plot [smooth] coordinates {(-2.5,-0.7) (-1.5,-0.5)(-1,0.6) (-0.3,0.6) (0,0.4)};
\draw [Rhodamine, thick, densely dashed] plot [smooth] coordinates {(5,-0.6)(3,-0.7)(2.5,-0.7) (1.5,-0.5)(1,0.6) (0.3,0.6) (0,0.4)};
\draw [violet, thick, densely dashed]  plot [smooth] coordinates {(1.75,0.78) (1.9,0.7)  (2.45,0.5)  (3.5,0.4) (4,-0.1)(5,0.5)};
\draw [violet,thick]  plot [smooth] coordinates {(0,0.4)(0.3,0.25)  (1,0.9)(1.4,0.9)(1.75,0.79)};
\node [above] at (1.2,0.95) {$\fil$};
\node [above] at (2.8,-0.75) {$\gilm$};
\draw [ thick,-] (1.75,-0.11) -- (1.75,0.11);
\node [above] at  (0.1,-0.4) {\small{{$0$}}}; 
\node [above] at  (1.75,-0.4) {\small{{$\varepsilon$}}}; 
\draw[gray,->](-1,-0.4) .. controls (-0.3,-1) and (0.3,-1) ..(1,-0.4);
\end{tikzpicture}
\caption{\footnotesize{The functions $\gilp$ (orange curve) and $\gilm$ (pink curve) are constructed from portions of the function $\fil$ (violet curve), by shifting and reflecting the latter about the $y$ axis, respectively. Vice versa, $\gilp$ and $\gilm$ determine $\fil$.}}\label{rys_l}
\end{figure}

{\bf Step 1. Rewriting \eqref{dodatek:1}.} Let us express relations \eqref{dodatek:1} in terms of $g$s.  To begin with,  by the definition of $\gilm$ and the first relation in \eqref{dodatek:1}, $\gilm (x) = \fil (-x) = \fjl (x), x \ge 0$, and so the definition of $\fil $ renders 
\begin{equation} \label{dodatek:4} \gilm = k_{j}\gjlp , \qquad  i=1,2,j=3-i, \end{equation}
where $k_j $ is the restriction of $f_j$ to the interval $[0,\eps]$, $j=1,2$. Also, because $\fil (\eps -x) = f_i(\eps - x)$ for $x \in [0,\eps]$ and $\fil (\eps -x) = \gilm (x-\eps)$ for $x >\eps $, we see that the function $[0,\infty)\ni x \mapsto \fil (\eps -x)$ coincides with $\hat k_i \gilm $ where $\hat k_i $ is the member of $C[0,\eps]$ given by $\hat k_i (x) = f_i(\eps - x), x \in [0,\eps]$. This together with \eqref{dodatek:4} shows that the function in question equals 
\[ \hat k_i (k_j \gjlp) = k_{ij} \gjlp, \qquad j=3-i, i =1,2, \]
where $k_{ij}\in C[0,2\eps]$ is the concatenation of $\hat k_i \in C[0,\eps]$ and $k_{j}\in C[0,\eps]$. (We note in passing that the manipulations on the left-hand side above are possible because $\hat k_i (\eps) = f_i(0)=f_j(0) = k_j(0)$ and for the same reason $k_{ij}$ belongs to $C[0,2\eps]$.) Hence, the last two lines of \eqref{dodatek:1} may be rewritten as 
\begin{align} 
\pepi \gir &= \pepi T_\eps f_i + (1-\pepi) k_{ij}\gjlp - (1-\pepi)\gilp, \nonumber \\
 \gir &= - T_\eps f_i + k_{ij}\gjlp + \gilp, \qquad i =1,2, j=3-i,\label{dodatek:5}
\end{align}
where $T_\eps f_i$ is the translation of $f_i$ by $\eps$ to the left.  To summarize: we have rewritten \eqref{dodatek:1} as \eqref{dodatek:4} and \eqref{dodatek:5}.  Also, the requirement that $\fir $ and $\fil$ belong to $C[-\infty,\infty]$ is now restated as the requirement that $\gilp, \gilm $ and $\gir$ belong to $C[0,\infty]$.

{\bf Step 2. Uniqueness of extensions.} 
By eliminating $\gir, i=1,2$ from the display above we obtain 
\begin{equation} \gilp = 2\pepi T_\eps f_i + (1-2\pepi) k_{ij} \gjlp, \qquad i=1,2,j=3-i. \label{dodatek:6} \end{equation}
This in turn, by \eqref{dodatek:2}, implies 
\begin{equation} \gilp = 2\pepi T_\eps f_i + 2p_{\eps,j} (1-2\pepi)(k_{ij}T_\eps f_j) + q_\eps (\overline{k_{ij}} \gilp), \quad i=1,2,j=3-i, \label{dodatek:7} \end{equation}
where $q_\eps$ has been introduced in \eqref{cocf:0} and $\overline{k_{ij}}$, a member of $C[0,4\eps]$, is the concatenation of $k_{ij}\in C[0,2\eps]$ and $k_{ji}\in C[0,2\eps]$. 

Let $X_{\eps,i}$ be the set of functions $g\in C[0,\infty]$ such that $g(0)=f_i(\eps)$. As a closed subset of the Banach space $C[0,\infty]$, $X_{\eps,i}$ is a complete metric space. Let us consider the map $O_{\eps,i}$ 
%defined in $X_{\eps,i}$, with values in the same space, 
that to a $g\in X_{\eps,i}$ assigns %another members of $X_{\eps,i}$ equal to 
\begin{equation}\label{dodatek:7a} O_{\eps,i} (g) \coloneqq 2\pepi T_\eps f_i + 2p_{\eps,j} (1-2\pepi) (k_{ij}T_\eps f_j) + q_\eps(\overline{k_{ij}}g).\end{equation}
Since, as it is easy to see,  the value of $O_{\eps,i} (g) $ at $0$ equals $f_i(\eps)$, $O_{\eps,i}$ maps $X_{\eps,i}$ into itself.  Moreover, for any $g^\flat$ and $g^\sharp$ in $X_{\eps,i}$, 
\begin{align*}
\| O_{\eps,i} (g^\flat) - O_{\eps,i} (g^\sharp)\| & = |q_\eps| \, \|  \overline{k_{ij}}g^\flat   - \overline{k_{ij}}g^\sharp \|  =  |q_\eps| \, \| g^\flat   - g^\sharp \|.   
\end{align*}
Now, by assumption, $|q_\eps|<1$, and so $O_{\eps,i}$ has a unique fixed point. But \eqref{dodatek:7} says that $\gilp$ is a fixed point of $O_{\eps,i}$. Hence, $\gilp, i=1,2$ are determined uniquely and, by \eqref{dodatek:4} and \eqref{dodatek:5}, so are $\gilm$ and $\gir, i=1,2$. Since $\gilm, \gilp $ and $\gir$, $i=1,2$ determine (and are determined by) $\fir$ and $\fil$, $i=1,2$, Step 2 is completed.

{\bf Step 3. Existence of extensions.} Now, let $\gilp$ be defined as the fixed point of $O_{\eps,i}$. Moreover, let the second line of \eqref{dodatek:5} define $\gir$ and \eqref{dodatek:4} define $\gilm, i=1,2$. In particular, since $O_{\eps,i}$ maps $X_{\eps,i}$ into itself, $\gilp (0) = f_i(\eps)$, implying that the right side of the unknown part of $\fil $ fits to the left part of $f_i$ at $\eps$; similar remarks apply to $\gilm$ and $\gir, i=1,2$.

Our goal will thus be reached once we establish the first line in \eqref{dodatek:5}; \eqref{dodatek:4} and the second line  in \eqref{dodatek:5} are satisfied automatically.  To this end, we note first that \eqref{dodatek:6} holds. Indeed, 
a short calculation shows that the right-hand side of \eqref{dodatek:6} is a fixed point of $O_{\eps,i}$ (because $\gjlp$ is a fixed point of $O_{\eps,j}$);  since $\gilp$ is the unique fixed point of $O_{\eps,i}$, this proves \eqref{dodatek:6}. With 
\eqref{dodatek:6} under our belt, we easily complete the argument: by substituting the right hand side of the second line of \eqref{dodatek:5} for $\gir$ in the first line, we realize that this line is nothing but \eqref{dodatek:6}.

{\bf Step 4. Relation \eqref{dodatek:3a}.} To prove \eqref{dodatek:3a} we recall that the Banach Fixed Point Theorem comes with an estimate of the distance between the fixed point and the iterates of the map involved (see e.g. \cite{jachymski}*{p. 5}).  In the case of $O_{\eps,i}$ (and the first iterate) the estimate reads 
\[ \| \gilp -  O_{\eps,i} g \|_{C[0,\infty]} \le \tfrac 1{1-|q_\eps|} \|g - O_{\eps,i} g\|_{C[0,\infty]} , \qquad g \in X_{\eps,i}. \] 
Next, we note that the definition of $O_{\eps,i}$ involves a linear combination of three functions: the first two of these are bounded in norm by $\|f\|_{C(S)}$, and the third is bounded in norm by the maximum of $\|g\|_{C[0,\infty]}$ and $\|f\|_{C(S)}$. Since the absolute values of the  coefficients of the combination add up to at most $5$, this shows that 
\( \|O_{\eps,i} g \|_{C[0,\infty]} \le 5 \|f\|_{C(S)} \)
as long as $\|g\|_{C[0,\infty]} = |f_i(\eps)|\le \|f\|_{C(S)}$, and we may take for example $g$ given by $g(x)= f_i(\eps), x \ge 0$. For such $g$, the display above renders 
\[ \|\gilp\|_{C[0,\infty]} \le  (5 + \tfrac 6{1-|q_\eps|}) \| f \|_{C(S)}. \]
This in turn, by \eqref{dodatek:4} and the second line of \eqref{dodatek:5}, yields 
\[  \|\gilm\|_{C[0,\infty]} \le  (5 + \tfrac 6{1-|q_\eps|}) \| f \|_{C(S)} \, \text{ and }  \,  \|\gir\|_{C[0,\infty]} \le  (11 + \tfrac {12}{1-|q_\eps|}) \| f \|_{C(S)}. \] 
The proof of \eqref{dodatek:3a} is thus completed, because $\|\fil\|_{C[-\infty,\infty]}$ does not exceed the maximum of $\|f_i\|_{C[0,\infty]}$, $\|\gilm\|_{C[0,\infty]}$ and $\|\gilp\|_{C[0,\infty]}$, and a~similar remark applies to the norm of $\fir$.  \end{proof}

With the information on extensions at our disposal, we can introduce a~family, say, $\{\overline C_\eps (t), t \in \R\}$, of linear operators in $\coss$ that will soon turn out to be a strongly continuous cosine family. To this end, given an $f=(f_1,f_2)\in \coss$ and a $t\in \R$, we construct the coordinates $(\overline C_\eps (t)f)_i, i =1,2$  of $\overline C_\eps (t)f $ by cutting the graphs of $C(t)\fil $ and $C(t)\fir$, which are members of $C[-\infty,\infty]$, into smaller pieces and then gluing some of them together according to the following prescription: 
\begin{equation}\label{dodatek:8} (\overline C_\eps (t)f)_i (x) = C(t)\fil(x) [0\le x\le \eps] +  C(t)\fir(x) [ x> \eps], \quad x\ge 0;  
\end{equation} 
as before $\{C(t), t\in \R\}$ is the basic cosine family defined in Section \ref{sec:cf}. 

To clear up any uncertainty as to continuity of $(\overline C_\eps (t)f)_i$, we note that the continuity of this function at $x=\eps$ is guaranteed (for any $t$) by the third line in \eqref{dodatek:1}. With that remark, the definition itself makes it amply clear that $ (\overline C_\eps (t)f)_i, i =1,2$ are elements of $C[0,\infty]$.  Moreover, the first line in \eqref{dodatek:1} implies that $(\overline C_\eps (t)f)_1 (0) = (\overline C_\eps (t)f)_2(0)$, that is, that $\overline C_\eps (t)f$ belongs to $\coss$, whereas \eqref{dodatek:1} in its entirety (when combined with uniqueness of extensions) reveals that the map assigning extensions to elements of $C(S)$ is linear. Finally, strong continuity of the basic cosine family implies that  $\{\overline C_\eps (t), t \in \R\}$ is a strongly continuous family too, and 
by \eqref{dodatek:3a}, 
\begin{equation}\label{dodatek:9} \|\overline C_\eps (t)f \| \le (11 + \tfrac {12}{1-|q_\eps|}) \|f\|, \qquad t\in \R, f \in \coss;  \end{equation}
this is to say that the operators $\overline C_\eps (t)$ are equibounded. 

\begin{proposition}\label{generowanie} $\{\overline C_\eps (t), t \in \R\}$ is a strongly continuous cosine family with generator $2\gep$.   \end{proposition}
\begin{proof} Let $f=(f_1,f_2)\in \coss$ be fixed, and let $\fil\in \cer$ and $\fir \in \cer$ be the related extensions. Also, for $\lam >0$, let 
\[\hil  \coloneqq \lam^{-1} \int_0^\infty \e^{-\lam t} C(t)\fil  \ud t \mqquad{and} \hir  \coloneqq \lam^{-1} \int_0^\infty \e^{-\lam t} C(t)\fir  \ud t. \]
As we know from Section \ref{sec:cf}, $\hil$ and $\hir$ belong to the domain of $\Delta$, and $\lam^2 \hil - \Delta \hil = \fil $ and $\lam^2 \hir - \Delta \hir = \fir $. 

Next, let 
\[ h=(h_1,h_2) \coloneqq \lam^{-1} \int_0^\infty \e^{-\lam t} \overline C_\eps (t) f  \ud t ; \]
the right-hand side is well-defined, because $\{\overline C_\eps (t), t \in \R\}$ is a strongly continuous family of equibounded operators. A routine argument establishes that 
\[ h_i (x) = \hil (x) [0\le x\le \eps] + \hir(x) [ x> \eps], \qquad i =1,2,\]
and this implies $\lam^2 h_i(x) - h_i''(x) = f_i(x), i =1,2$ for all $x\ge 0$ save perhaps $x=\eps$. But $h_i$ is a continuous function and so is $f_i$; hence, $h_i''(\eps-) = h_i''(\eps+)$. Next, the second line in \eqref{dodatek:1} implies that $\pepi h_i' (\eps+) = (1-\pepi) h_i'(\eps-), i =1,2$. Indeed, as a straightforward calculation shows, the last relation comes down to the fact that the Laplace transforms of the functions featured on both sides of the second line in \eqref{dodatek:1} coincide. Similarly, the first line in \eqref{dodatek:1} implies that $\frac 12 h_1'(0) + \frac 12 h_2'(0)=0$. It follows that $h$ belongs to the domain of $\gep$, and the differential equations established above turn into $\lam^2 h - 2\gep h = f$.

Considering that $f$ was chosen arbitrarily, we have proved that  
\[ \lam (\lam^2 - 2\gep)^{-1} f = \int_0^\infty \e^{-\lam t} \overline C_\eps (t) f  \ud t, \qquad f \in \coss, \lam >0.  \]
This relation, however, implies (see \cite[Proposition 3.14.4]{abhn}) that  $\{\overline C_\eps (t), t \in \R\}$ is a cosine family and that $2 \gep$ is its generator, as desired. 
\end{proof}

Finally, Proposition \ref{generowanie} implies  Theorem \ref{dodatek} (a). Indeed, since $2\gep$ is the generator of $\{\overline C_\eps (t), t \in \R\}$, $\gep$ is the generator of  $\{C_\eps (t), t \in \R\}$
where $C_\eps (t) \coloneqq  \overline C_\eps (\frac {\sqrt 2}2 t), t \in \R$. It is worth noting, however, that Proposition \ref{generowanie} says more than  Theorem \ref{dodatek} (a): it provides a semi-explicit form of the cosine family in question.  This form will be crucially used in the proof of Theorem \ref{dodatek} (b). 

\subsection{Proof of part (b)} 
So far, we have suppressed notationally the fact that extensions $\fil$ and $\fir$ depend on $\eps$ and so do $\gir, \gilm$ and $\gilp$. Since this section is devoted to convergence as $\eps \to 0$, we cannot continue in this fashion: from now on, we replace $\fil$ and $\fir$ with $\file$ and $\fire$, respectively; similarly, we write $\gire, \gilme$ and $\gilpe$ instead of $\gir, \gilm$ and $\gilp$. In particular, \eqref{dodatek:8} takes the form \begin{equation}\label{dod:1} (\overline C_\eps (t)f)_i (x) = C(t)\file (x) [0\le x\le \eps] +  C(t)\fire (x) [ x> \eps], \quad x\ge 0.   
\end{equation}

Before continuing, we recall the following well-known extension of the Banach Fixed Point Theorem (see e.g. \cites{zmarkusem2b}).  Suppose $(X,d)$ is a complete metric space, $q_0\in (0,1)$ is a given number,  and  $U_\eps , \eps >0$ is a family of maps such that 
\[ d(U_\eps \mathbf x, U_\eps \mathbf  y) \le q_0 d( \mathbf x, \mathbf y ), \qquad \mathbf x, \mathbf y \in X\]
and for each $\mathbf x\in X$, the limit $U\mathbf  x\coloneqq \grae U_\eps \mathbf x $ exists. Then, $U$ is a contraction with constant $q_0$, and its unique fixed point is the limit of the fixed points of $U_\eps$s. 

Since for each $i=1,2$, $\gilpe, \eps >0$ are fixed points of the maps $O_{\eps,i}, \eps>0$ defined in \eqref{dodatek:7a}, the result spoken of above seems to be a natural tool for proving convergence of $\gilpe, \eps >0$. However, in applying it we encounter a somewhat unexpected difficulty: the spaces where $O_{\eps,i}$ are defined vary with $\eps$. In particular,  $O_{\eps,i}g, g\in \cerp$, is rarely defined for all sufficiently small $\eps$, and thus it makes no sense to speak of $\grae O_{\eps,i}g$. Fortunately, our case is not hopeless as all the spaces involved are subsets of $C[0,\infty]$ and there is a natural way of `shifting' elements from one space to another. This gives room for a straightforward modification of the proof of the convergence result spoken of above, a modification that  renders convergence of $\gilpe, \eps >0$. Details are contained in the following lemma (see the proof of (ii)  in particular). 

\begin{lemma}\label{acik} Let $f=(f_1,f_2)\in \coss$ and $i=1,2$ be fixed. Let, as before, $X_{\eps,i}, \eps\ge 0$ be the set of functions $g\in C[0,\infty]$ such that $g(0)=f_i(\eps)$. \begin{itemize} \item [(i) ]
For $g\in \cerp$, let $\mathpzc c_\eps = \mathpzc c_\eps (g) \coloneqq f_i(\eps) - g(0), \eps \ge 0$. Then $g+\mathpzc c_\eps $ belongs to $X_{\eps,i}$ and 
\( \grae O_{\eps,i} (g+\mathpzc c_\eps) = O_i (g) \) 
where 
\[ O_i(g) \coloneqq 2p_i f_i + 2p_j (1-2p_i)f_j + q (g+c_0), \quad j=3-i. \]
\item [(ii) ] As restricted to $X_{0,i}$, $O_i$ is a contraction with constant $|q|<1$, and thus has a unique fixed point, say, $\gilpz$. Moreover, 
$\grae \gilpe =\gilpz$.  
\item [(iii) ] The limits $\filz \coloneqq \grae \file $ and  $\firz \coloneqq \grae \fire $ exist and are given as follows 
\end{itemize} 
\begin{align} \filz (x)&  \coloneqq \begin{cases} \frac {p_i}{p_1 + p_2 - 2p_1 p_2} f_i (x) + \frac {p_j - 2p_1p_2}{p_1 + p_2 - 2p_1 p_2} f_j(x), & x\ge 0, \\
 \frac {p_j}{p_1 + p_2 - 2p_1 p_2} f_j (-x) + \frac {p_i - 2p_1p_2}{p_1 + p_2 - 2p_1 p_2} f_i(-x), & x<0, \end{cases}\nonumber \\
 \firz (x) & \coloneqq \begin{cases} f_i(x) & x\ge 0, \\
 \frac {p_i - p_j}{p_1 + p_2 - 2p_1 p_2} f_i (-x) + \frac {2p_j - 2p_1p_2}{p_1 + p_2 - 2p_1 p_2} f_j(-x), & x<0, \end{cases}\label{dod:2} \end{align}
  \( i =1,2, j=3-i. \)

\end{lemma} 

\begin{proof} We begin by noting that $k_{ij}$ and $\overline {k_{ij}}$ featured in \eqref{dodatek:7a} belong to $C[0,2\eps]$ and $C[0,4\eps]$, respectively, and in particular, despite appearances caused by shortened notation, vary with $\eps$.    

{\bf Proof of (i)} Since the semigroup of translations to the left is strongly continuous in $\cerp$, the first term on the right of \eqref{dodatek:7a} converges, as $\eps \to 0$, to $2p_if_i$. By the same token, \[ \max_{x> 2 \eps} |T_\eps f_j(x) - f_j(x)|\le \|T_\eps f_j - f_j\| \to 0, \text{ as } \eps \to 0.\] Moreover, since $f_1(0)=f_2(0)$, given $\delta >0$ one can choose $\eps_0$ so that $|k_{ij} (x) - k_{ij} (2\eps)|\le \frac \delta 2 $ for $x\in [0,2\eps]$,  $\eps \le\eps_0$. Continuity of $f_j$ at $0$ implies at the same time that by making $\eps_0$ smaller if necessary we can guarantee that $|f_j(x) -f_j(\eps)| \le \frac \delta 2$ for $x\in [0,2\eps],$ $\eps \le \eps_0$. Because of $k_{ij} (2\eps )=f_j(\eps)$ it follows that $\max_{0\le x\le 2\eps} | k_{ij} (x) - f_j(x)| \le \delta$ for $\eps \le \eps_0$. This, when combined with the display above shows that the second term on the right of \eqref{dodatek:7a} converges to $2p_j (1-2p_i)f_j$, $\delta>0$ being chosen arbitrarily. 

Similarly, we show that $\overline { k_{ij}} (g+\mathpzc c_\eps)$ converges to $g+c_0$. For, %for $x\ge 4\eps $,
 \begin{align*} \max_{x\ge 4\eps} | \overline {k_{ij}} (g+\mathpzc c_\eps)(x) - g(x) - c_0| & = \max_{x\ge 4\eps} |g(x-4\eps) -g(x) +\mathpzc c_\eps -c_0|\\
 & \le \max_{x\ge 0} |g(x)-g(x+4\eps)| + |f_i(\eps) - f_i(0)|
  \end{align*} 
converges to $0$ as $\eps \to 0$, because $g$ is uniformly continuous. At the same time, given $\delta >0$, one can find $\eps_0$ so small that, as long as $\eps \le \eps_0$, 
\begin{itemize}
\item [(a) ] $|\overline{k_{ij}} (x) - \overline {k_{ij}}(0)| \le \frac \delta 3, x \in [0,4\eps]$,
\item [(b) ] $|g(x) - g(0)| \le \frac \delta 3, x \in [0,4\eps]$, and
\item [(c) ] $|f_i(\eps) - f_i(0)| \le \frac \delta 3.$ 
\end{itemize}
Therefore, $\max_{0\le x \le 4\eps} | \overline {k_{ij}} (g+\mathpzc c_\eps)(x) - g(x) - c_0|\le \delta$, and the claim follows as in the previous case. This completes the proof of (i). 

{\bf Proof of (ii)} For $g\in X_{0,i}$, the third term in the definition of $O_i$ reduces to $qg,$ and thus the first part of (ii) immediately follows.

Moreover, by assumption, there is a constant $q_0\in (0,1)$ such that for all sufficiently small $\eps$, $|q_\eps| \le q_0$. Since $\gilpe$ is a fixed point of $O_{\eps,i}$, for such $\eps$ we may write
\begin{align*}
\|\gilpz - \gilpe \| & \le \|\gilpz - O_{\eps,i} (\gilpz +\mathpzc c_\eps)\| \\& \phantom{=} + \| O_{\eps,i} (\gilpz +\mathpzc c_\eps)- O_{\eps,i} (\gilpe) \| \\ & \le \|\gilpz - O_{\eps,i} (\gilpz +\mathpzc c_\eps)\| + q_0 \| \gilpz +\mathpzc c_\eps- \gilpe \| \\ &\le 
 \|\gilpz - O_{\eps,i} (\gilpz +\mathpzc c_\eps)\| + q_0 |\mathpzc c_\eps| + q_0 \| \gilpz - \gilpe \|, 
\end{align*}
where $\mathpzc c_\eps = \mathpzc c_\eps (\gilpz)$. Hence, 
\begin{align*}
\|\gilpz - \gilpe \| & \le \tfrac 1{1-q_0} \|\gilpz - O_{\eps,i} (\gilpz +\mathpzc c_\eps)\| + \tfrac{q_0}{1-q_0} |\mathpzc c_\eps|. 
\end{align*}
The first term on the right hand side here converges to $0$ by (i), and so does the second because $\gilpz$ belongs to $X_{0,i}$. 
This completes the proof of (ii).  

{\bf Proof of (iii)} The fixed point $\gilpz$ of $O_{i}$ can be calculated explicitly: 
\[ \gilpz = \tfrac {p_i}{p_1 + p_2 - 2p_1 p_2} f_i + \tfrac {p_j - 2p_1p_2}{p_1 + p_2 - 2p_1 p_2} f_j, \qquad i =1,2, j=3-i. \]
Also, using (ii) and \eqref{dodatek:4} (recall that $\gilm$ and $\gilp$ featured there are really $\gilme $ and $\gilpe$, respectively), and arguing as above we see that $\gilmz \coloneqq \grae \gilme = g_{0,j,+}^{\mathpzc l}$. A similar argument shows then that $\filz \coloneqq \grae \file $ exists and establishes the first part of \eqref{dod:2}. The second part of \eqref{dodatek:5} yields therefore 
\( \girz \coloneqq  \grae \gire = - f_i + g_{0,j,+}^{\mathpzc l} + \gilpz .\) This leads, as above, to the rest of the assertion. %\eqref{dodatek:5}. 
\end{proof}
 
We are finally in a position to complete the proof of Theorem \ref{dodatek}(b).

\begin{proof}[Proof of Theorem \ref{dodatek} \emph{(b)}] We claim that, see \eqref{dod:1}, for each $i=1,2$
\begin{equation}\label{dod:3} \grae \sup_{t \in \R} \|(\overline C_\eps (t) f)_i - C(t) \firz \| = 0, \qquad f \in \coss.  \end{equation}
Interestingly, despite the fact that the limit does not involve $\filz$,  the proof uses the following relation between $\filz$ and $\firz$, a direct consequence of \eqref{dod:2}: 
\begin{equation}\label{dod:4} \filz (t) + \filz(-t) = \firz (t) + \firz (-t), \qquad t \in \R. \end{equation}

Here are the details. For $x>\eps$ and $t \in \R$,
\( (\overline C_\eps (t) f)_i (x) - C(t)\firz (x) = C(t) \fire (x) - C(t)\firz (x). \) Hence, the supremum over such $x$ of the absolute value of the left-hand side does not exceed $\|C(t)\fire - C(t)\firz\|\le \|\fire - \firz\|$, and this converges to $0$, as $\eps \to 0$ (uniformly in $t\in \R$) by Lemma \ref{acik} (iii). 

Next, for $x\in [0,\eps]$, 
\[ (\overline C_\eps (t) f)_i (x) - C(t)\firz (x) = \pol (\file (x+t) +\file (x-t) - \firz (x+t) - \firz (x-t)). \]
Let $\delta >0$ be given. By Lemma \ref{acik} (iii), there is an $\eps_0$ so small that $\|\filz - \file\| \le \frac \delta 3$, whenever $\eps \le \eps_0$.  
Also, since $\filz$ and $\firz$, being members of $\cer$, are uniformly continuous, by making $\eps_0$ smaller if necessary, we can make sure that 
\begin{align*}
| \filz (x+t) + \filz (x-t) - \filz (t) - \filz (-t) |& \le \tfrac \delta 3,   \\
| \firz (x+t) + \firz (x-t) - \firz (t) - \firz (-t) |& \le \tfrac \delta 3, \qquad x\in [0,\eps], t\in \R. \end{align*}
By \eqref{dod:4}, it follows that, for $\eps \le \eps_0$,  $\max_{x\in [0,\eps]} | (\overline C_\eps (t) f)_i (x) - C(t)\firz (x)| \le \delta$. Since $\delta >0$ is arbitrary, we see that  $\grae  \sup_{t\in \R} \max_{x\in [0,\eps]} | (\overline C_\eps (t) f)_i (x) - C(t)\firz (x)|=0$. This completes the proof of \eqref{dod:3}.

Finally, in terms of $\wt \alpha_i, i =1,2$ introduced in \eqref{wojoj}, the second line of the definition of $\firz$ can be written as $2\wt \alpha_i f_i(-x) + 2\wt \alpha_j f_j(-x) - f_i (-x)$, $j=3-i$. This allows recognizing $(\firz )_{i=1,2}\in (\cer)^2$ as the (unique) extension of $f\in \coss$ shaped by skew Brownian motion's transmission conditions with parameters $\wt \alpha_i, i =1,2$, see \eqref{ext_skew}. 
In other words, the limit cosine family of $\overline C_\eps$ is generated by  $2\mathfrak H_{\wt \alpha,0}$.   The remark made at the end of Section \ref{sec:popa} completes the entire proof.  \end{proof}

 \section*{Acknowledgements} A.Pilipenko was  partially supported by the Swiss National Science
Foundation, grants No. IZRIZ0\_226875, No. 200020\_214819, No. 10009199, No. 200020\_192129.

%\bibliographystyle{plain}
%\bibliography{bibliografiakopia}
%\bibliography{../../../../bibliografia}

\begin{bibdiv}
\begin{biblist}

\bib{abhn}{book}{
      author={Arendt, W.},
      author={Batty, C. J.~K.},
      author={Hieber, M.},
      author={Neubrander, F.},
       title={Vector-{V}alued {L}aplace {T}ransforms and {C}auchy {P}roblems},
   publisher={Birkh{\"a}user},
     address={Basel},
        date={2001},
}

\bib{bass}{book}{
      author={Bass, R.~F.},
       title={Stochastic processes},
      series={Cambridge Series in Statistical and Probabilistic Mathematics},
   publisher={Cambridge University Press, Cambridge},
        date={2011},
      volume={33},
        ISBN={978-1-107-00800-7},
         url={https://doi.org/10.1017/CBO9780511997044},
}

\bib{laplace}{article}{
      author={Bobrowski, A.},
       title={The {W}idder--{A}rendt theorem on inverting of the {L}aplace
  transform, and its relationships with the theory of semigroups of operators},
        date={1997},
     journal={Methods Funct. Anal. Topology},
      volume={3},
      number={4},
       pages={1\ndash 39},
}

\bib{kosinusy}{article}{
      author={Bobrowski, A.},
       title={Generation of cosine families via {L}ord {K}elvin's method of
  images},
        date={2010},
     journal={J. Evol. Equ.},
      volume={10},
      number={3},
       pages={663\ndash 675},
}

\bib{tombatty}{incollection}{
      author={Bobrowski, A.},
       title={Families of operators describing diffusion through permeable
  membranes},
        date={2015},
   booktitle={Operator semigroups meet complex analysis, harmonic analysis and
  mathematical physics\textrm , {A}rendt, {W}., {C}hill, {R}., {T}omilov,~{Y}.,
  eds.},
      series={Operator Theory, Advances and Applications},
      volume={250},
   publisher={Birkh{\"a}u\-ser},
       pages={87\ndash 105},
}

\bib{knigazcup}{book}{
      author={Bobrowski, A.},
       title={Convergence of {O}ne-{P}arameter {O}perator {S}emigroups. {I}n
  {M}odels of {M}athematical {B}iology and {E}lsewhere},
   publisher={Cambridge University Press, Cambridge},
        date={2016},
}

\bib{zwojtkiem}{article}{
      author={Bobrowski, A.},
      author={Chojnacki, W.},
       title={Cosine families and semigroups really differ},
        date={2013},
        ISSN={1424-3199},
     journal={J. Evol. Equ.},
      volume={13},
      number={4},
       pages={897\ndash 916},
         url={http://dx.doi.org/10.1007/s00028-013-0208-0},
}

\bib{abtk}{article}{
      author={Bobrowski, A.},
      author={Komorowski, T.},
       title={Diffusion approximation for a simple kinetic model with
  asymmetric interface},
        date={2022},
     journal={J. Evol. Equ.},
      volume={22:42},
}

\bib{zmarkusem2b}{article}{
      author={Bobrowski, A.},
      author={Kunze, M.},
       title={Irregular convergence of mild solutions of semilinear equations},
        date={2019},
        ISSN={0022-247X},
     journal={J. Math. Anal. Appl.},
      volume={472},
      number={2},
       pages={1401\ndash 1419},
         url={https://doi.org/10.1016/j.jmaa.2018.11.082},
}

\bib{abap}{misc}{
      author={Bobrowski, A.},
      author={Pilipenko, A.},
       title={Analytic and stochastic description of {B}rownian motions on star
  graphs},
        date={2025},
         url={https://arxiv.org/abs/2511.20039},
        note={arXiv 2511.20039},
}

\bib{onportenko}{misc}{
      author={Bobrowski, A.},
      author={Pilipenko, A.},
       title={On {P}ortenko's approximation of skew {B}rownian motion},
        date={2026, To appear in SIAM Journal on Mathematical Analysis, see
  also arXiv 2601.20440},
         url={https://arxiv.org/abs/2601.20440},
}

\bib{bobrat}{article}{
      author={Bobrowski, A.},
      author={Ratajczyk, E.},
       title={{On one-parameter continuous family of pairs of complementary
  boundary conditions}},
        date={2022},
     journal={Studia Mathematica},
      volume={266 (1)},
       pages={81\ndash 92},
}

\bib{aberman}{article}{
      author={Bobrowski, A.},
      author={Ratajczyk, E.},
       title={Pairs of complementary transmission conditions for {B}rownian
  motion},
        date={2024},
        ISSN={0025-5831,1432-1807},
     journal={Math. Ann.},
      volume={388},
      number={4},
       pages={4317\ndash 4342},
         url={https://doi.org/10.1007/s00208-023-02613-x},
}

\bib{abernachr}{article}{
      author={Bobrowski, A.},
      author={Ratajczyk, E.},
       title={Approximation of skew {B}rownian motion by snapping-out
  {B}rownian motions},
        date={2025},
     journal={Mathematische Nachrichten},
      volume={298},
      number={3},
       pages={829\ndash 848},
      eprint={https://onlinelibrary.wiley.com/doi/pdf/10.1002/mana.202400179},
         url={https://onlinelibrary.wiley.com/doi/abs/10.1002/mana.202400179},
}

\bib{fromsotows}{misc}{
      author={Bobrowski, A.},
      author={Ratajczyk, E.},
       title={From snapping out {B}rownian motions to {W}alsh's spider
  processes on star-like graphs},
        date={2025},
         url={https://arxiv.org/abs/2406.16800},
        note={To appear in {D}ocumenta {M}athematica, arXiv 2406.16800},
}

\bib{ethier}{book}{
      author={Ethier, S.~N.},
      author={Kurtz, T.~G.},
       title={{M}arkov {P}rocesses. {C}haracterization and {C}onvergence},
   publisher={Wiley},
     address={New York},
        date={1986},
}

\bib{goldi}{article}{
      author={Goldstein, J.~A.},
       title={On the convergence and approximation of cosine functions},
        date={1974},
     journal={Aequationes Math.},
      volume={11},
       pages={201\ndash 205},
}

\bib{goldstein}{book}{
      author={Goldstein, J.~A.},
       title={Semigroups of {L}inear {O}perators and {A}pplications},
   publisher={Oxford University Press},
     address={New York},
        date={1985},
}

\bib{bookandrey}{book}{
      author={Iksanov, A.},
      author={Marynych, A.},
      author={Pilipenko, A.},
      author={Samoilenko, I.},
       title={Locally {P}erturbed {R}andom {W}alks},
   publisher={Springer},
        date={2025},
}

\bib{ito}{book}{
      author={It\^{o}, K.},
      author={McKean{,}~Jr., H.~P.},
       title={Diffusion {P}rocesses and {T}heir {S}ample {P}aths},
   publisher={Springer},
     address={Berlin},
        date={1996},
        note={Repr. of the 1974 ed.},
}

\bib{jachymski}{article}{
      author={Jachymski, J.},
      author={J\'o\'zwik, I.},
      author={Terepeta, M.},
       title={The {B}anach {F}ixed {P}oint {T}heorem: selected topics from its
  hundred-year history},
        date={2024},
     journal={Rev. Real Acad. Cienc. Exactas Fis. Nat. Ser. A-Mat.},
      volume={118},
         url={https://doi.org/10.1007/s13398-024-01636-6},
}

\bib{kallenbergnew}{book}{
      author={Kallenberg, O.},
       title={Foundations of {M}odern {P}robability},
     edition={2},
   publisher={Springer},
        date={2002},
}

\bib{konishi}{article}{
      author={Konishi, Y.},
       title={Cosine functions of operators in locally convex spaces},
        date={1971/1972},
     journal={J. Fac. Sci. Univ. Tokyo Sect. IA Math.},
      volume={18},
       pages={443\ndash 463},
}

\bib{koralov2026}{misc}{
      author={Koralov, L.},
      author={Tao, X.},
       title={A criterion for convergence of diffusion processes on graphs},
        date={2026, ArXiv 2609.22520},
         url={https://arxiv.org/abs/2609.22520},
}

\bib{kostrykin2012}{article}{
      author={Kostrykin, V.},
      author={Potthoff, J.},
      author={Schrader, R.},
       title={Brownian motions on metric graphs},
        date={2012Sep},
        ISSN={1089-7658},
     journal={Journal of Mathematical Physics},
      volume={53},
      number={9},
       pages={095206},
         url={http://dx.doi.org/10.1063/1.4714661},
}

\bib{lejayskew}{article}{
      author={Lejay, A.},
       title={On the constructions of the skew {B}rownian motion},
        date={2006},
     journal={Probab. Surv.},
      volume={3},
       pages={413\ndash 466},
         url={https://doi.org/10.1214/154957807000000013},
}

\bib{makhnomem}{article}{
      author={Makhno, S.~Y.},
       title={One-dimensional stochastic equations in layered media with
  semi-permeable barriers},
        date={2016},
        ISSN={0926-6364,1569-397X},
     journal={Random Oper. Stoch. Equ.},
      volume={24},
      number={3},
       pages={165\ndash 171},
         url={https://doi.org/10.1515/rose-2016-0013},
      review={\MR{3545916}},
}

\bib{rosen}{book}{
      author={Marcus, M.~B.},
      author={Rosen, J.},
       title={Markov {P}rocesses, {G}aussian {P}rocesses, and {L}ocal {T}imes},
      series={Cambridge Studies in Advanced Mathematics},
   publisher={Cambridge University Press},
        date={2006},
}

\bib{withvillmart}{misc}{
      author={Pilipenko, A.},
      author={Vilmart, G.},
       title={Limit theorems for one-dimensional diffusions: reflected, skew,
  sticky, skew-sticky, and snapping-out limits},
         how={in preparation},
	note={Arxiv},
        date={2026},
}

\bib{portenko}{incollection}{
      author={Portenko, N.~I.},
       title={Generalized diffusion processes},
        date={1976},
   booktitle={Proceedings of the {T}hird {J}apan-{USSR} {S}ymposium on
  {P}robability {T}heory ({T}ashkent, 1975), {Lecture Notes in Math. Vol.
  550}},
      series={Lecture Notes in Math.},
      volume={Vol. 550},
   publisher={Springer, Berlin-New York},
       pages={500\ndash 523},
}

\bib{portenkokniga}{book}{
      author={Portenko, N.~I.},
       title={Generalized {D}iffusion {P}rocesses},
      series={Translations of Mathematical Monographs},
   publisher={American Mathematical Society},
        date={1990},
      volume={83},
}

\bib{dapratos}{article}{
      author={Prato, G.~Da},
      author={Sinestrari, E.},
       title={Differential operators with nondense domain},
        date={1987},
     journal={Ann. Scuola Norm. Sup. Pisa Cl. Sci. {\rm (4)}},
      volume={14},
       pages={285\ndash 344},
}

\bib{rogers}{book}{
      author={Rogers, L. C.~G.},
      author={Williams, D.},
       title={Diffusions, {M}arkov {P}rocesses and {M}artingales, {V}ol. 1,
  {F}oundations},
   publisher={Cambridge University Press},
     address={Cambridge},
        date={2000},
}

\end{biblist}
\end{bibdiv}

\def\cprime{$'$}\def\cprime{$'$}\def\cprime{$'$}\def\polhk#1{\setbox0=\hbox{#1}{\ooalign{\hidewidth
  \lower1.5ex\hbox{`}\hidewidth\crcr\unhbox0}}}\ifx \cedla \undefined \let
  \cedla = \c \fi\ifx \cyr \undefined \let \cyr = \relax \fi\ifx \cprime
  \undefined \def \cprime {$\mathsurround=0pt '$}\fi\ifx \prime \undefined \def
  \prime {'} \fi\def\polhk#1{\setbox0=\hbox{#1}{\ooalign{\hidewidth
  \lower1.5ex\hbox{`}\hidewidth\crcr\unhbox0}}}

 \end{document}